\documentclass[11pt]{amsart}
\usepackage{amssymb,amsmath,amsthm,mathrsfs}

\usepackage{a4wide}

\usepackage{enumerate,paralist}

\usepackage{tikz-cd}

\numberwithin{equation}{section}
\newtheorem{theorem}{Theorem}[section]
\newtheorem{proposition}[theorem]{Proposition}
\newtheorem{lemma}[theorem]{Lemma}
\newtheorem{corollary}[theorem]{Corollary}
\theoremstyle{definition}
\newtheorem{definition}[theorem]{Definition}
\newtheorem{remark}[theorem]{Remark}

\usepackage{amsmath,amsfonts,amscd,amssymb,epsf,mathrsfs,color,enumerate}

\begin{document}

\title{Maximal discs of Weil-Petersson class in $\mathbb{A}\mathrm{d}\mathbb{S}^{2,1}$}
\author{Jinsung Park}
\address{School of Mathematics\\ Korea Institute for Advanced Study\\
207-43\\ Hoegi-ro 85\\ Dong\-daemun-gu\\ Seoul 130-722\\
Korea }
\email{jinsung@kias.re.kr}

\thanks{2020 Mathematics Subject Classification
	30F60, 32G15, 53C50, 53A10 }

\date{\today}

\begin{abstract}{We introduce maximal discs of Weil-Petersson class in the 3-dimensional Anti-de Sitter space $\mathbb{A}\mathrm{d}\mathbb{S}^{2,1}$, whose parametrization space can be identified with
 the cotangent bundle $T^*T_0(1)$ of Weil-Petersson universal Teichm\"uller space $T_0(1)$.
We prove that the Mess map defines a symplectic diffeomorphism from $T^*T_0(1)$ to $T_0(1)\times T_0(1)$,  with respect to the canonical symplectic form
on $T^*T_0(1)$ and the difference of pullbacks of the Weil-Petersson symplectic forms from each factor of $T_0(1)\times T_0(1)$. 
Furthermore, we show that the functional given by
the anti-holomorphic energies of the induced Gauss maps associated with maximal discs of Weil-Petersson class serves as a K\"ahler potential for the restriction of the canonical symplectic form to
certain submanifolds $T_0(1)^\pm \subset T^*T_0(1)$, which bijectively parametrize the space of maximal discs of Weil-Petersson class  in
$\mathbb{A}\mathrm{d}\mathbb{S}^{2,1}$.}
\end{abstract}

\maketitle

\bibliographystyle{plain}

\section{Introduction}

In their seminal work \cite{TT06}, Takhtajan and Teo introduced a distinguished subspace of the universal Teichm\"uller space $T(1)$,  endowed with
a natural Hermitian structure. This subspace, denoted $T_0(1)$, is known as the \emph{Weil-Petersson} universal Teichm\"uller space, as it carries a rich geometric
structure induced by the Weil-Petersson inner product and its associated symplectic 2-form $\omega_{\mathrm{WP}}$. Moreover, 
in \cite{TT06}, they introduced the \emph{universal Liouville action} $S$ for  elements in $T_0(1)$,  and  proved that $S$  serves as a K\"ahler potential for $\omega_{\mathrm{WP}}$, satisfying the identity
\begin{equation}\label{e:TT06}
\partial \bar{\partial} S = -2 \sqrt{-1}\, \omega_{\mathrm{WP}} \qquad \text{over}\quad T_0(1),
\end{equation}
where $\partial$ and $\bar{\partial}$ denote the holomorphic and anti-holomorphic derivatives on $T_0(1)$. 
The identity \eqref{e:TT06} can be regarded as a universal analogue of the results previously established for  classical Liouville
actions on Teichm\"uller spaces of Riemann surfaces, as presented in \cite{TZ88a}, \cite{TZ88b}, \cite{TT03}, \cite{PTT17}, and \cite{PT18}.

An element of  $T_0(1)$ can be represented by a quasi-circle in the complex plane satisfying  a specific condition, and such curves are referred to as
 \emph{Weil-Petersson curves}.  Since their introduction in \cite{TT06},
these curves have attracted significant attention and have become a central object of study  in various  areas of mathematics. In particular,
Wang, in \cite{W19}, introduced 
the Loewner energy $I_L$ for Weil-Petersson curves,  defined via the energy of the driving function naturally associated with the Shramm Loewner evolution, and proved
the identity 
\begin{equation}\label{e:W19}
S =\pi I_L \qquad \text{over}\quad  T_0(1).
\end{equation}
This identity is particularly striking, as it equates two quantities defined through entirely different frameworks -
one arising from Teichm\"uller theory and the other from Schramm-Loewner theory.

In a recent work \cite{B20}, Bishop undertook a detailed study of Weil-Petersson curves from both
analytical and geometric perspectives.  In particular,  he introduced an invariant associated with a minimal surface in the hyperbolic 3-space $\mathbb{H}^3$ that bounds a given Weil-Petersson curve on the boundary $\partial \overline{\mathbb{H}}^3$.
However,  a limitation of this approach is the non-uniqueness of such bounding minimal surfaces in $\mathbb{H}^3$,
which complicates efforts to establish  a direct connection between the invariant of the bounding minimal surface in $\mathbb{H}^3$  and either the universal Liouville action or,  equivalently, the Loewner energy of the Weil-Petersson curve.

Motivated by the previously described developments in \cite{TT06}, \cite{W19}, and \cite{B20}, we consider maximal surfaces in the Anti de Sitter 3-space $\mathbb{A}\mathrm{d}\mathbb{S}^{2,1}$,
which can be viewed as Lorentzian analogues  of  minimal surfaces in $\mathbb{H}^3$.  
More precisely,  we consider a maximal discs of \emph{Weil-Petersson class} 
in  $\mathbb{A}\mathrm{d}\mathbb{S}^{2,1}$, whose boundary at infinity
is the graph of a quasisymmetric homeomorphism of $\mathbb{S}^1$ representing an element in $T_0(1)$.
Such a bounding maximal disc exists uniquely for each element in $T_0(1)$ by the result of \cite{BS10}. 
Given a maximal disc   $\Sigma\subset \mathbb{A}\mathrm{d}\mathbb{S}^{2,1}$,  we consider a conformal embedding 
$$
\sigma:\mathbb{D} \to  \mathbb{A}\mathrm{d}\mathbb{S}^{2,1}
$$
such that $\sigma(\mathbb{D})=\Sigma$ where $\mathbb{D}$ is the unit disc in the complex plane. 
Then,  associated with the Gauss map of this maximal conformal embedding into  $\mathbb{A}\mathrm{d}\mathbb{S}^{2,1}$, there exists a pair of harmonic maps  
$$
F_\pm:\mathbb{D}\to \mathbb{D}.
$$
Here harmonicity encodes specific geometric structures on $\mathbb{D}$,
a detailed discussion  of which will be given later.
In general, the integrals over $\mathbb{D}$ of the anti-holomorphic energy densities of the harmonic maps
$F_\pm$  diverge.  However,  in Proposition \ref{t:total-curv},  we show that
the anti-holomorphic energies of the harmonic maps $F_\pm$ are finite, provided  that 
the maximal disc $\Sigma \subset \mathbb{A}\mathrm{d}\mathbb{S}^{2,1} $ satisfies
the Weil-Petersson condition. As we will state more precisely, the anti-holomorphic energy of $F_\pm$ plays the role of the Liouville action $S$, or equivalently the Loewner energy $I_L$. 

To formulate this result, we require a geometric framework based on the symplectic geometry of the holomorphic cotangent bundle $T^*T(1)$ and so called $\mathrm{Mess}$ map.
We observe that $T^*T(1)$ parametrizes
all conformal embeddings 
into $\mathbb{A}\mathrm{d}\mathbb{S}^{2,1}$,  as we will elaborate on in Section  \ref{s:maximal-WP}.   In this context,
we introduce  the  \emph{Mess map},  following the original construction by Mess in \cite{M07}:  
\begin{equation}\label{e:Mess-whole}
\mathrm{Mess}: T^*T(1) \to T(1)\times T(1).
\end{equation}
 In this paper, we prove that the restriction
$$
\mathrm{Mess}:T^*T_0(1)\to T_0(1)\times T_0(1)
$$
is a symplectic diffeomorphism  with respect to the canonical symplectic form $\omega_{\mathrm{C}}$
on $T^*T_0(1)$ and the difference of  pullbacks of Weil-Petersson symplectic forms from each factor of $T_0(1)\times T_0(1)$. Here, the canonical symplectic form $\omega_C$ is defined as the imaginary part of
the complex canonical symplectic form $\omega_{\mathbb{C}}$.  See \eqref{e:symplectic-def} for its precise definition. 
 With this background in place, our main result concerning the $\mathrm{Mess}$ map is stated as follows:

\begin{theorem}\label{t:main thm1} The map 
$$
\mathrm{Mess}:T^*T_0(1)\to T_0(1)\times T_0(1)
$$ is a symplectic diffeomorphism. That is,
\[
\omega_{\mathrm{C}} = -\mathrm{Mess}^*_+(\omega_{\mathrm{WP}}) + \mathrm{Mess}^*_-(\omega_{\mathrm{WP}}),
\]
where $\mathrm{Mess}_\pm:= \pi_{\pm}\circ \mathrm{Mess}$,  and $\pi_\pm$ denote the projection maps from $T_0(1)\times T_0(1)$
onto the first or second factors, respectively.
\end{theorem}

A related result to Theorem \ref{t:main thm1} was presented in \cite{KS2013}, where a similar claim 
was made for the $\mathrm{Mess}$ map at the origins of  $T^*T(1)$ and  $T(1)\times T(1)$.
While the statement  in  \cite{KS2013} was formulated in terms of the Weil-Petersson symplectic form on $T(1)$, the precise nature of such a structure  on $T(1)$ remains to be clarified,  especially given that the construction of  the  Weil-Petersson symplectic form $\omega_{\mathrm{WP}}$ on $T_0(1)$ relies essentially on the underlying
Hilbert manifold structure of $T_0(1)$ in \cite{TT06}.  This same Hilbert manifold structure on $T_0(1)$ also plays a crucial role in the proof of Theorem \ref{t:main thm1}.  See also Remark \ref{r:bounded} for a related implication of the Hilbert structure on a fiber of the tangent bundle $TT_0(1)$.

Let us denote by $T_0(1)^+$, $T_0(1)^-$ the submanifolds of  the cotangent bundle $T^*T_0(1)$,  defined as the inverse images of  $T_0(1)\times\{0\}$, $\{0\}\times T_0(1)$, respectively,
under the $\mathrm{Mess}$ map.  These submanifolds are real symplectic submanifold of the complex manifold $T^*T_0(1)$,
equipped with the restriction of $\omega_{\mathrm{C}}$. 
Each of  these submanifolds can also be interpreted as the image of a 
(non-holomorphic) section of $T^*T_0(1)$.  Consequently, we can  endow $T_0(1)^\pm$ with a holomorphic structure by identifying them 
with $T_0(1)$ via these sections. This phenomenon appears similarly  in the derivation of the equality \eqref{e:TT06}  in \cite{TT06}, where the corresponding submanifold of $T^*T_0(1)$ arises from the (non-holomorphic) section defined by the Schwarzians of  conformal welding factors of  given quasisymmetric homeomorphisms in $T_0(1)$. 
Further  discussions on these are provided in Remark \ref{r:Kahler submanifold}. 
As mentioned in Remark  \ref{r:parametrization},  the submanifolds $T_0(1)^\pm$ also admit a geometric interpretation: they provide a bijective
parametrization of the space of maximal discs of Weil-Petersson class in $\mathbb{A}\mathrm{d}\mathbb{S}^{2,1}$.   
We then obtain the following  result:

\begin{theorem} \label{t:main thm2} 
The anti-holomorphic energy $E(F_\pm)$ of the induced harmonic maps $F_\pm$ defines a finite-valued functional on $T^*T_0(1)$. Moreover,  the following identity holds:
\begin{equation}
2\, \partial\bar{\partial}  E  = \mp \sqrt{-1}\, i^*_\pm\, \omega_{\mathrm{C}} \qquad \text{over} \quad T_0(1)^\pm.
\end{equation}
Here $i_\pm :T_0(1)^\pm\to T^*T_0(1)$ denote the inclusion maps.
\end{theorem}

Combining equations  \eqref{e:TT06}, \eqref{e:W19} with Theorems \ref{t:main thm1} and  \ref{t:main thm2}, we obtain the following:

\begin{corollary}\label{c:potentials}
The following  identity holds over $T_0(1)^{\pm}$:
\begin{equation}\label{e:potentials}
4\,\partial\bar{\partial}  E = i^*_\pm \mathrm{Mess}_{\pm}^* (\partial\bar{\partial} S) =\pi\, i^*_\pm \mathrm{Mess}_{\pm}^*(\partial \bar{\partial} I_L).
\end{equation}
\end{corollary}

\begin{remark}
Theorem \ref{t:main thm2}  states that $T_0(1)^\pm$  admits a K\"ahler structure, where the complex structure is inherited via its identification with $T_0(1)$, and the K\"ahler form is given by $i^* \omega_{\mathrm{C}}$, up to a constant.  Corollary \ref{c:potentials} further shows that this K\"ahler structure on $T_0(1)^\pm$ is symplectically equivalent to the 
Weil-Petersson structure on $T_0(1)$. The discrepancy between these K\"ahler structures stems from
the intrinsic differences of the anti-holomorphic energy functional $E$ and the Liouville action $S$, or equivalently the Loewner energy $\pi I_L$.  For a more detailed discussion on
the relation between $E$ and $S=\pi I_L$, we refer the reader to Remark \ref{r:relation-E-S}.
\end{remark}

Here is an explanation of structure of this paper. In Section \ref{s:basic}, we review foundational material on the universal Teichm\"uller space $T(1)$ and
the Weil-Petersson universal Teichm\"uller space $T_0(1)$.
Section \ref{s:maximal-WP} introduces the notion of the maximal discs of Weil-Petersson class in  $\mathbb{A}\mathrm{d}\mathbb{S}^{2,1}$ and investigates their fundamental properties.
Section \ref{s:variation} and \ref{s:anti-holomorphic} are devoted to the proofs of the main theorems,  employing variational techniques  which played crucial roles in \cite{TZ88a}, \cite{TZ88b}, \cite{TT03},
\cite{TT06}, \cite{PTT17},  \cite{PT18}, and \cite{Park21}.  Finally, the appendix provides a brief overview of the 3-dimensional Anti de Sitter space  $\mathbb{A}\mathrm{d}\mathbb{S}^{2,1}$.

{{\bf{Acknowledgments}} 
I would like to express my deep gratitude to Graham Andrew Smith for pointing out a gap in the proof of Proposition \ref{t:total-curv} and for bringing to my attention Theorem B in his paper \cite{DMSST26}, which plays a crucial role in completing the proof.}

\section{Universal Teichm\"uller Space}\label{s:basic}
%\label{sec:BM}
This section provides a brief introduction to the universal Teichmüller space and the Weil-Petersson universal Teichmüller space. For further details, we refer to Chapter 1 of \cite{L87} and Chapter 16 of \cite{GL00} for the universal Teichmüller space, and Chapter 1 of \cite{TT06} for the Weil-Petersson universal Teichmüller space.

\subsection{Universal Teichm\"uller space}\label{ss:UTC}
Let $\mathrm{QS}(\mathbb{S}^1)$ denote the group of the quasisymmetric homeomorphism of the circle $\mathbb{S}^1$. The \textit{universal Teichm\"uller space} is then defined by
\begin{equation}\label{e:def-QS}
T(1):= \mathrm{Mob}(\mathbb{S}^1) \backslash \mathrm{QS}(\mathbb{S}^1)
\end{equation}
where $\mathrm{Mob}(\mathbb{S}^1)\cong \mathrm{PSL}(2,\mathbb{R})$ is the subgroup of the M\"obius transformation group $\mathrm{PSL}(2,\mathbb{C})$ that preserves $\mathbb{S}^1$ and
 acts on $\widehat{\mathbb{C}}=\mathbb{C}\cup\{\infty\}$.

Let $\mathbb{D}=\{\, z\in \mathbb{C}\, | \, |z|<1\, \}$ denote the open unit disk and let $\mathbb{D}^*=\{\, z\in\mathbb{C} \, | \, |z|>1\, \}$ be its exterior in $\widehat{\mathbb{C}}$. 
Denote by $L^\infty(\mathbb{D})$ and $L^\infty(\mathbb{D}^*)$ the complex Banach
spaces of bounded Beltrami differentials on $\mathbb{D}$ and $\mathbb{D}^*$ respectively. Let $L^\infty(\mathbb{D})_1$ denote the unit ball in $L^\infty(\mathbb{D})$. For a given  Beltrami differential $\mu\in L^\infty(\mathbb{D})_1$, we extend it to $\mathbb{D}^*$ by the reflection
\begin{equation}
\mu(z)= \overline{\mu\Big(\frac{1}{\bar{z}}\Big)} \frac{z^2}{\bar{z}^2}, \qquad \text{for} \quad z\in\mathbb{D}^*.
\end{equation}
Let $w_{\mu}:\widehat{\mathbb{C}}\to \widehat{\mathbb{C}}$
 be the solution of the Beltrami differential equation 
\begin{equation}\label{e:beltrami-d-e}
\partial_{\bar{z}} w_{\mu} = \mu\, \partial_z w_{\mu}
\end{equation}
with fixed points $1,-1, -\sqrt{-1}$. Then $w_{\mu}$ preserves $\mathbb{S}^1$ and satisfies $w_{\mu}|_{\mathbb{S}^1} \in \mathrm{QS}(\mathbb{S}^1)$. Conversely, by extension theorem of Beurling-Ahlfors, any quasisymmetric homeomorphism in $\mathrm{QS}(\mathbb{S}^1)$ can be extended to a quasiconformal homeomorphism $w_{\mu}$ of $\mathbb{D}$ 
for some $\mu\in L^\infty(\mathbb{D})_1$. This leads to the following description of the universal Teichm\"uller space:,
\begin{equation}\label{e:projection}
T(1)=L^\infty(\mathbb{D})_1/ \sim
\end{equation}
where $\mu \sim \nu$ if and only if $w_{\mu}|_{\mathbb{S}^1} = w_{\nu}|_{\mathbb{S}^1}$. We denote the equivalence class of $\mu$ by $[\mu]\in T(1)$. The space $T(1)$ admits a unique structure of a complex Banach manifold such that the projection map 
\[
\mathrm{P}: L^\infty(\mathbb{D})_1\to T(1)
\]
is a holomorphic submersion. The differential of $\mathrm{P}$ at the origin
\[
D_0\mathrm{P}:L^\infty(\mathbb{D})\to T_{[0]}T(1)
\]
is a complex linear surjection onto the holomorphic tangent space of $T(1)$. The kernel of $D_0 \mathrm{P}$ is the subspace $\mathcal{N}(\mathbb{D})$ of infinitesimally trivial Beltrami differentials.

For a given Beltrami differential $\mu\in L^\infty(\mathbb{D})_1$, extend it to be zero on $\mathbb{D}^*$. Let $w^{\mu}$ be the unique solution to the Beltrami differential equation 
$$
\partial_{\bar{z}} w^{\mu} = \mu\, \partial_z w^{\mu}
$$ 
with the normalization $w^\mu(0)=0$, $(w^\mu)_{z}(0)=1$ and $(w^\mu)_{zz}(0)=0$. Then $w^{\mu}$ is conformal on $\mathbb{D}^*$. 
This leads to the following characterization of the universal Teichm\"uller space:
\begin{equation}\label{e:projection-1}
T(1)=L^\infty(\mathbb{D})_1/ \sim
\end{equation}
where $\mu \sim \nu$ if and only if $w^{\mu}|_{\mathbb{D}} = w^{\nu}|_{\mathbb{D}}$.
This characterization is equivalent to the one in \eqref{e:projection} since $w_\mu|_{\mathbb{S}^1}=w_\nu|_{\mathbb{S}^1}$ if and only if
$w^{\mu}|_{\mathbb{D}} = w^{\nu}|_{\mathbb{D}}$.

Now, we define the Bers embedding of $T(1)$ into the complex Banach space
\begin{equation*}
A_\infty(\mathbb{D}^*)=\{\, \phi:\mathbb{D}^* \to \mathbb{C}\, | \, \text{holomorphic},\ \mathrm{sup}_{\mathbb{D}^*} |\phi| e^{-\psi} < \infty \, \}
\end{equation*}
where $e^{\psi}$ denotes the hyperbolic density function on $\mathbb{D}^*$.
For a holomorphic map $f$ on an open domain in $\widehat{\mathbb{C}}$,
the Schwarzian of $f$ is defined by
\begin{equation}
\mathcal{S}(f)= \left(\frac{f_{zz}}{f_z}\right)_z -\frac12  \left(\frac{f_{zz}}{f_z}\right)^2.
\end{equation}
For every $\mu\in L^\infty(\mathbb{D})_1$, the holomorphic function $\mathcal{S}(w^\mu|_{\mathbb{D}^*})$
belongs to $A_\infty(\mathbb{D}^*)$ by Kraus-Nehari inequality. The Bers embedding is then defined by
\begin{equation}\label{e:Bers-embedding}
\beta([\mu])= \mathcal{S}(w^\mu|_{\mathbb{D}^*})\in A_\infty(\mathbb{D}^*).
\end{equation}
This embedding is a holomorphic map between complex Banach manifolds. 

The Banach space of bounded harmonic Beltrami differential on 
$\mathbb{D}$ is defined by
\[
\Omega^{-1,1} (\mathbb{D})=\{\, \mu\in L^\infty(\mathbb{D})\, | \, \mu= e^{-\psi} \bar{\phi}, \, \phi\in A_\infty(\mathbb{D})\, \}
\]
where $e^{\psi}$ denotes the hyperbolic density function on $\mathbb{D}$ and $A_\infty(\mathbb{D})$ is defined analogously to $A_\infty(\mathbb{D}^*)$.
The decomposition
\begin{equation}\label{e:space-decomp}
L^\infty(\mathbb{D})=\Omega^{-1,1}(\mathbb{D})\oplus \mathcal{N}(\mathbb{D})
\end{equation}
identifies
the holomorphic tangent space $T_{[0]}T(1)\cong L^\infty(\mathbb{D})/\mathcal{N}(\mathbb{D})$ at the origin in $T(1)$ as 
\begin{equation}
T_{[0]} T(1) \cong \Omega^{-1,1}(\mathbb{D}).
\end{equation}
The complex linear mapping $D_0\beta$ induces
an isomorphism $\Omega^{-1,1}(\mathbb{D})\cong A_\infty(\mathbb{D}^*)$ between the holomorphic tangent spaces
to $T(1)$ and $A_\infty(\mathbb{D}^*)$ at the origin.

The unit ball $L^\infty(\mathbb{D})_1$ carries a group structure induced by
the composition of quasiconformal maps. The group law $\lambda=\nu\star \mu^{-1}$ is defined via
\[
w_\lambda= w_\nu\circ w_\mu^{-1}.\]
The explicit formula of the group law is given by 
\begin{equation}\label{e:composition}
w_{\mu}^* (\lambda):=\lambda\circ w_{\mu}\frac{\overline{\partial_z w_\mu}}{{ \partial_z w_\mu}} = \frac {\nu-\mu}{1-\nu\bar{\mu}} \, .
\end{equation} 
For $\mu\in L^\infty(\mathbb{D})_1$, using this group structure,
we define the right translation $R_\mu$ on $L^\infty(\mathbb{D})_1$. The induced right translations on $T(1)$
\begin{equation}
R_{[\mu]}: T(1) \longrightarrow T(1), \qquad [\lambda]\mapsto [\lambda\star\mu]
\end{equation}
are biholomorphic automorphisms of $T(1)$.
Consequently, the differential
\[
D_0R_{[\mu]}: T_{[0]} T(1)\to T_{[\mu]} T(1)
\]
is a complex linear isomorphism between the holomorphic tangent spaces, leading to the identification 
$T_{[\mu]}T(1) \cong \Omega^{-1,1}(\mathbb{D})$.

For $\mu\in L^\infty(\mathbb{D})_1$,
let $U_\mu$ be the image of the ball of radius 2 in $A_{\infty}(\mathbb{D}^*)$ under the map $
h_\mu^{-1}=\mathrm{P}\circ R_\mu\circ \Lambda$
where $\Lambda$ is the inverse of $D_0\beta$. 
Then the maps 
\[
h_{\mu\nu}:= h_\mu\circ h_\nu^{-1}: h_\nu(U_\mu\cap U_\nu) \to h_\mu(U_\mu\cap U_\nu)
\]
are biholomorphic as maps on the Banach space $A_\infty(\mathbb{D}^*)$. The structure of $T(1)$ as a complex Banach manifold, modeled on the Banach space $A_\infty(\mathbb{D}^*)$, is explicitly described by the complex-analytic atlas given by the open covering
\begin{equation*}
T(1)= \bigcup_{\mu\in L^\infty(\mathbb{D})_1} U_\mu
\end{equation*}
with coordinate maps $h_\mu$ and the transition maps $h_{\mu\nu}$. 
Complex coordinates on $T(1)$, defined by the coordinate charts $(U_\mu, h_\mu)$, are referred to as \emph{Bers coordinates}.
For every $\nu\in\Omega^{-1,1}(\mathbb{D})$, let $\phi=D_0\beta(\nu)$ and define a holomorphic vector field $\frac{\partial}{\partial \epsilon_\nu}$ on $U_0$ by setting
\[
Dh_0\Big(\frac{\partial}{\partial \epsilon_\nu}\Big)=\phi
\]
at all points in $U_0$.  At every point $[\mu]\in U_0$, identified with the corresponding harmonic Beltrami differential $\mu$, the vector
field $\frac{\partial}{\partial \epsilon_\nu}$ in terms of the Bers coordinates of $U_\mu$ correspond to
\[
\tilde{\phi} =D_\mu h_\mu \Big(\frac{\partial}{\partial \epsilon_\nu}\Big) =
\Big( D_\mu h_\mu(D_\mu h_0)^{-1}\Big)(\phi)= D_0(\beta\circ\mathrm{P})(D_\mu R_{\mu}^{-1}(\Lambda(\phi))).
\]
Using identification $\Omega^{-1,1}(\mathbb{D})\cong A_\infty(\mathbb{D}^*)$, provided by $D_0\beta$,
\begin{equation}\label{e:local-vector}
\frac{\partial}{\partial\epsilon_\nu}\Big\vert_{\mu} =D_0\mathrm{P} (D_{\mu} R_{\mu}^{-1}(\nu)) = D_0\mathrm{P}(R(\nu,\mu)),
\end{equation}
where
\begin{equation}\label{e:pullback}
R(\nu,\mu):=D_{\mu} R_{\mu}^{-1}(\nu)= \Big(\frac{\nu}{1-|\mu|^2} \frac{(w_{\mu})_z}{(\bar{w}_{\mu})_{\bar{z}}}\Big)\circ w_\mu^{-1}.
\end{equation}

\subsection{Weil-Petersson universal Teichm\"uller space}\label{ss:Hilbert}
Consider the space
\begin{equation*}
A_2(\mathbb{D}^*)=\Big\{\, \phi :\mathbb{D}^* \to \mathbb{C}\, | \, \text{holomorphic}, \int_{\mathbb{D}^*} |\phi|^2 e^{-\psi} \, d^2z < \infty \, \Big \} \subset A_\infty(\mathbb{D}^*)
\end{equation*}
where $d^2z=dxdy$ for $z=x+\sqrt{-1}y$.
Let $\mathcal{O}(\mathbb{D})_1$  denote the subgroup of $L^\infty(\mathbb{D})_1$ generated by $\mu\in \Omega^{-1,1}(\mathbb{D})$ with $||\mu||_\infty <\delta$ where $\delta$ is a positive real number satisfying
condition in Corollary 2.6 in \cite{TT06}. For each  $\mu\in\mathcal{O}(\mathbb{D})_1$,
let $V_\mu\subset U_\mu \subset T(1)$ be the image under the map $h_\mu^{-1}=\mathrm{P}\circ R_\mu\circ \Lambda$
of the open ball of radius $\sqrt{\pi/3}$  centered at the origin in $A_2(\mathbb{D}^*)$.
Define 
\begin{equation}\label{e:coordinate-map}
\tilde{h}_\mu=h_\mu\Big|_{V_\mu}: V_\mu\to A_2(\mathbb{D^*}).
\end{equation}
Now, consider the covering
\[
T(1)=\bigcup_{\mu\in\mathcal{O}(\mathbb{D})_1} V_\mu
\]
with the coordinate maps $\tilde{h}_\mu:V_\mu\to A_2(\mathbb{D}^*)$ and the transition maps
\[
\tilde{h}_{\mu\nu}=\tilde{h}_\mu\circ \tilde{h}_\nu^{-1}: \tilde{h}_\nu(V_\mu\cap V_\nu)\to \tilde{h}_\mu(V_\mu\cap V_\nu).
\]

 By Theorem 2.10 in \cite{TT06}, the above covering gives $T(1)$ the structure of a complex Hilbert manifold, modeled on the Hilbert
space $A_2(\mathbb{D}^*)$.  However,  $T(1)$ is not connected with respect to the topology induced by the Hilbert manifold structure.

For the Hilbert space of harmonic Beltrami differentials on $\mathbb{D}$,
\begin{equation*}
H^{-1,1}(\mathbb{D})=\{\, \mu \in L^\infty(\mathbb{D})\, | \, \mu= e^{-\psi} \bar{\phi}, \, \phi\in A_2(\mathbb{D})\, \} \subset \Omega^{-1,1} (\mathbb{D}),
\end{equation*}
and for $[\mu]\in T(1)$, let $D_0 R_{[\mu]} (H^{-1,1}(\mathbb{D}))$ be the subspace of the tangent space
$T_{[\mu]}T(1)= D_0R_{[\mu]} (\Omega^{-1,1}(\mathbb{D}))$, which is equipped with a Hilbert space structure isomorphic to $H^{-1,1}(\mathbb{D})$.
Let $\mathcal{D}_T$ be the distribution on $T(1)$, defined by the assignment
\[
T(1) \ni [\mu] \mapsto D_0R_{[\mu]}(H^{-1,1}(\mathbb{D})) \subset T_{[\mu]} T(1).
\]
By Theorem 2.3 in \cite{TT06}, for every $[\mu]\in T(1)$, the linear mapping
\[
D_0(\beta\circ R_{[\mu]}): H^{-1,1}(\mathbb{D}) \to A_2(\mathbb{D}^*)
\]
is a topological isomorphism.
Moreover, by Theorem 2.13 in \cite{TT06}, the Bers embedding 
$$\beta:T(1)\to \beta(T(1))\subset A_\infty(\mathbb{D}^*)$$
is a biholomorphic mapping of Hilbert manifolds. As a result, the distribution $\mathcal{D}_T$
on $T(1)$ is integrable. The integral manifolds correspond to the components $(\phi+A_2(\mathbb{D}^*))\cap \beta(T(1))$.
For every $[\mu]\in T(1)$, we denote by $T_{[\mu]}(1)$ the component of the Hilbert manifold $T(1)$ containing $[\mu]$.
The Hilbert manifold $T_{[\mu]}(1)$ is the integral manifold of the
distribution $\mathcal{D}_T$ passing through $[\mu]\in T(1)$. In particular, the component of the origin $0\in T(1)$ is denoted by
$T_0(1)$, and is called the \emph{Weil-Petersson universal Teichm\"uller space}.

The Weil-Petersson metric on the Hilbert manifold $T_0(1)$ is a Hermitian metric defined 
by the Hilbert space inner product on tangent space, which is identified with
the Hilbert space $H^{-1,1}(\mathbb{D})$ via right translation. Thus, the Weil-Petersson metric is a right invariant metric on $T_0(1)$,
defined at the origin of $T_0(1)$ by
\begin{equation}
\langle \mu, \nu \rangle = \int_{\mathbb{D}} \mu \bar{\nu}\, e^{\psi} \, d^2z, \qquad \text{for}\quad \mu,\nu\in H^{-1,1}(\mathbb{D})=T_0T_0(1).
\end{equation}
For every $\mu\in H^{-1,1}(\mathbb{D})$, there corresponds a vector field $\frac{\partial}{\partial \epsilon_\mu}$ over $V_0$.
For every $\kappa\in V_0$, we define the inner product
\[
g_{\mu\bar{\nu}}(\kappa)= \Big\langle \frac{\partial}{\partial \epsilon_\mu}\Big\vert_{\kappa}, \frac{\partial}{\partial \epsilon_\nu}\Big\vert_{\kappa} \Big\rangle_{WP}
=\int_{\mathbb{D}} D_0\mathrm{P}(R(\mu,\kappa)) \overline{D_0\mathrm{P}(R(\nu,\kappa))}\, e^{\psi} \, d^2z
\]
where $R(\mu,\kappa)$ is given in \eqref{e:pullback}.
The Weil-Petersson metric extends to other charts $V_\mu$ by right translations.

\section{Maximal discs of Weil-Petersson class}\label{s:maximal-WP}

In this section, we introduce and study maximal discs of Weil-Petersson class in the three-dimensional Anti de Sitter space $\mathbb{A}\mathrm{d}\mathbb{S}^{2,1}$. We interpret
the space of  the conformal embeddings of these maximal discs as the cotangent bundle $T^*T_0(1)$ of the Weil-Petersson universal Teichm\"uller space $T_0(1)$, and show that it can be identified with $T_0(1)\times T_0(1)$ via the map originally defined by Mess in \cite{M07}. For basic terminology related to the geometry of $\mathbb{A}\mathrm{d}\mathbb{S}^{2,1}$,
we refer to the appendix of this paper.  For a more extensive introduction to Anti-de Sitter geometry, we refer to \cite{BS20}.

\subsection{Maximal discs and Gauss maps}

We fix the conformal structure on the standard unit disc $\mathbb{D}\subset \mathbb{C}$
and denote it by $\mathbb{D}_w$ where the coordinate map $w$ is the identity on $\mathbb{D}$. 
Given a Beltrami differential $\mu$ representing an element of the universal Teichm\"uller space $T(1)$, we denote
by $\mathbb{D}_z$ the corresponding unit disc equipped with the coordinate map $z$, which is the quasi-conformal map $w_{\mu}: \mathbb{D} \to \mathbb{D}$ introduced in the subsection \ref{ss:UTC}.   
Hence, for instance,  $w_\mu:\mathbb{D}_z \to\mathbb{D}_w$ is a conformal map, whereas  $w_\mu:\mathbb{D}_w\to\mathbb{D}_w$ is not.

For a given $\mathbb{D}_z$, we denote its spacelike conformal embedding into the three-dimensional Anti de Sitter space $\mathbb{A}\mathrm{d}\mathbb{S}^{2,1}$ by
\begin{equation}\label{e:embbeding-rho}
\sigma: \mathbb{D}_z \to \mathbb{A}\mathrm{d}\mathbb{S}^{2,1}.
\end{equation}
 The conformal embedding $\sigma$ of the unit disc $\mathbb{D}_z$ is called a \emph{maximal} if the mean curvature identically vanishes, that is, $H_\sigma=0$ on 
 the spacelike surface $\Sigma:=\sigma(\mathbb{D})$. The image of a conformal embedding with this property is also refer to as a \emph{maximal disc}.
 From now on, we may regard the pullback of a function (or tensor) on $\Sigma$ via $\sigma$ as a function (or tensor) on $\mathbb{D}_z$.  In particular, we can interpret the mean curvature over $\Sigma$ as a function on $\mathbb{D}_z.$

For a spacelike conformal embedding $\sigma: \mathbb{D}_z \to \mathbb{A}\mathrm{d}\mathbb{S}^{2,1}$,  the associated Gauss map,  introduced in \eqref{e:A-Gauss},  is given by
$$G=(G_+,G_-): \mathbb{D}_z \to \mathbb{D}_w\times \mathbb{D}_w. 
$$ 

By Proposition 3.1 and Theorem 3.3 of \cite{AAW00}, we have:

\begin{proposition}\label{p:basics-G}
For a maximal conformal embedding $\sigma: \mathbb{D}_z \to \mathbb{A}\mathrm{d}\mathbb{S}^{2,1}$, the following hold:
\begin{enumerate}
\setlength{\itemsep}{7pt}
\item{ Each component of the Gauss map 
$G_\pm:\mathbb{D}_z\to \mathbb{D}_w$ is a harmonic map,}
\item{The pullback metric of $g_{\mathbb{A}\mathrm{d}\mathbb{S}^{2,1}}$ under $\sigma$, which is the fundamental form $I$, is given by
$$ \sigma^*(g_{\mathbb{A}\mathrm{d}\mathbb{S}^{2,1}})=I=e^\phi|dz|^2:= e^{\psi\circ G_\pm} \big|(G_\pm)_z\big|^2 |dz|^2, $$
where $e^{\psi(w)}|dw|^2$ is the hyperbolic metric on $\mathbb{D}_w$.}
\item{The Hopf differential of $G_\pm$  is given by
\[
\Phi(G_\pm)=e^{\psi\circ G_\pm} (G_\pm)_z (\overline{G_\pm})_z\, dz^2,
\]
and it satisfies the relation $\Phi(G_+)=-\Phi(G_-)$.}
\end{enumerate}

\end{proposition}

Note that the harmonic map equation for $G_{\pm}$ is given by
\begin{equation}\label{e:harmonic map}
(G_\pm)_{z\bar{z}} + (\psi_w\circ G_\pm) (G_\pm)_z (G_\pm)_{\bar{z}} =0.
\end{equation}
As evident from this equation,
the harmonicity condition on a Riemann surface depends 
on the metric structure  of the target Riemann surface
and the conformal structure of the source Riemann surface, rather than the metic structure of the source itself. 
The underlying fact for Proposition \ref{p:basics-G} is that the pullback metric $G_\pm^*(e^{\psi(w)} |dw|^2)$ decomposes as:
\begin{equation}\label{e:pullback-metric-pm}
   G_\pm^*(e^{\psi(w)} |dw|^2) = (1+|\mu_{G_\pm}|^2) I +\Phi(G_\pm) + \overline{\Phi(G_\pm)},
   \end{equation}
   where $\mu_{G_\pm}$ denotes the Beltrami differential of $G_\pm$.
The trace part corresponds to $I$ up to the conformal factor $(1+|\mu_{G_\pm}|^2)$, and the off-trace parts are given by the Hopf differential $\Phi(G_\pm)$, as expected.

The second fundamental form of $\sigma: \mathbb{D}_z\to \mathbb{A}\mathrm{d}\mathbb{S}^{2,1}$ is given by:
\begin{equation}\label{e:I-II}
II=\frac{\sqrt{-1}}{2}\big(\Phi(G_+)-\overline{\Phi(G_+)}\big).
\end{equation}
For further details on these equalities,
we refer to Section 3 of \cite{AAW00} and Section 5 of \cite{KS2013}.

By the Gauss equation for a maximal conformal embedding $\sigma: \mathbb{D}_z\to \mathbb{A}\mathrm{d}\mathbb{S}^{2,1}$, we obtain the following relation:
\begin{equation}\label{e:Gauss}
2\phi_{z\bar{z}}= e^\phi - e^{-\phi} |\Phi(G_\pm)|^2.
\end{equation}
Note that this equation can also be derived from the harmonic map equation for $G_\pm$ given in \eqref{e:harmonic map}.
We also have the following proposition:

\begin{proposition}
The Gaussian curvature $K_\phi$ of the metric $I=e^{\phi}|dz|^2$ on $\mathbb{D}_z$
is given by
\begin{equation}\label{e:Gaussian-curv}
K_\phi:=-2\phi_{z\bar{z}}\, e^{-\phi} = -1+|\mu_{G_\pm}|^2.
\end{equation}
\end{proposition}

\begin{proof}
For simplicity,  let $G=G_\pm$ during the proof. By the definition of $e^{\phi}$ and \eqref{e:harmonic map}, 
we first obtain
\[
\phi_z=(\psi_w\circ G) G_z +\frac{G_{zz}}{G_z},
\]
and
\[
\phi_{z\bar{z}}=(\psi_{w\bar{w}}\circ G) \big(|G_z|^2-|G_{\bar{z}}|^2\big).
\]
Using the Liouville equation for $e^\psi$, that is, $\psi_{w\bar{w}}=1/2 e^{\psi}$, we get
\[
\phi_{z\bar{z}}=\frac12 e^{\psi\circ G}\big(|G_z|^2-|G_{\bar{z}}|^2\big).
\]
Thus, the Gaussian curvature follows as
\[
K_{\phi}=-2\phi_{z\bar{z}}\,e^{-\phi}=-e^{\psi\circ G}\big(|G_z|^2-|G_{\bar{z}}|^2\big)\cdot e^{-\psi\circ G}|G_z|^{-2} = -1+|\mu_G|^2.
\]
%This implies the nonvanishing of $K_\phi$ under the condition $|\mu_G| <1$.
\end{proof}

By the Gauss equation for a maximal conformal embedding  $\sigma: \mathbb{D}_z\to \mathbb{A}\mathrm{d}\mathbb{S}^{2,1}$,
the Gaussian curvature is given by
$$K_\phi=-1 +\kappa^2
$$ 
where $\pm \kappa$ are the principal curvatures of the maximal disc
$\Sigma\subset  \mathbb{A}\mathrm{d}\mathbb{S}^{2,1}$.
Consequently, we obtain 
\begin{equation}\label{e:mu-kappa}
|\mu_{G_\pm}|^2= \kappa^2.
\end{equation}

The pullback metrics in \eqref{e:pullback-metric-pm} by $G_\pm$ are two different hyperbolic metrics, which
induce new conformal structures on $\mathbb{D}_z$, denoted by $\mathbb{D}_{z_{\pm}}$, respectively. 
The identity map between $\mathbb{D}_z$ and $\mathbb{D}_{z_\pm}$ equipped with the hyperbolic metric
can then be interpreted as a harmonic map, with its Hopf differential given by
$\pm \Phi=\Phi(G_\pm)$, respectively. We denote these maps by
$$
F_\pm:\mathbb{D}_z\to \mathbb{D}_{z_\pm}.
$$
Furthermore, by the construction, we have the following:

\begin{proposition}\label{p:G-F-same}
The metric density $e^{\phi(z)}$ on $\mathbb{D}_z$ satisfies the relation
\begin{equation}\label{e:G-F}
e^\phi=e^{\psi\circ G_\pm}|(G_\pm)_z|^2= e^{\psi\circ F_\pm}|(F_\pm)_z|^2
\end{equation}
where $e^\psi$ represents the hyperbolic metric density on $\mathbb{D}_w$, $\mathbb{D}_{z_\pm}$, respectively. Moreover, the Beltrami differentials remain invariant, that is,
\begin{equation}\label{e:quad-Beltrami}
\mu_{G_\pm}= \pm \overline{\Phi} e^{-\phi}=\mu_{F_\pm}.
\end{equation}
\end{proposition}

By Proposition \ref{p:G-F-same}, the pair $(F_+,F_-)$ shares fundamental properties with the Gauss map $(G_+,G_-)$. Thus, we refer to $(F_+,F_-)$ the \emph{induced Gauss map}.  

Combining the above constructions for both cases of $F_\pm$, we obtain the following commutative diagram
\begin{equation}\label{e:CD1} 
\begin{tikzcd}
\large
\mathbb{D}_{z_+}   &    \ \mathbb{D}_z  \arrow[l, swap,  "F_+"]  \arrow[r,  "F_-"]  & \mathbb{D}_{z_-}  \\
& \mathbb{D}_w \arrow [ul, "z_+"] \arrow [u, "z"] \arrow [ur, swap, "z_-"]
\end{tikzcd} 
\end{equation}
where $z,z_+,z_-$ denote the quasi-conformal maps from $\mathbb{D}_w$ to $\mathbb{D}_z$, $\mathbb{D}_{z_+}$, $\mathbb{D}_{z_-}$ respectively. Furthermore, we assume that all maps $z,z_\pm$ and $F_\pm$ are \emph{normalized}, meaning
they preserve the points $1,-1,-\sqrt{-1}$.
%Two points represented by $(\mathbb{D}, z_{\pm})$ can be considered to have the same distance from  the point represented by $(\mathbb{D}, z)$ in $T(1)$ in the sense that $|%\Phi(F_+)|=|\Phi(F_-)|$.

By Theorem 4.1 of \cite{AAW00}, Theorem 1.10 of \cite{BS10}, and the previous construction,  we obtain
the following proposition:

\begin{proposition}\label{p:equiv-rel} For a fixed conformal structure on $\mathbb{D}_z$, there exist  one-to-one correspondences between the following:
\vspace{7pt}
\begin{enumerate} 
\setlength{\itemsep}{7pt}
\item{A maximal conformal embedding $\sigma:\mathbb{D}_z \to \mathbb{A}\mathrm{d}\mathbb{S}^{2,1}$
with a complete induced metric and Gaussian curvature that is negative and bounded away from zero.}
\item{An orientation-preserving minimal Lagrangian diffeomorphism $G: \mathbb{D}_w \to \mathbb{D}_w$ where $\mathbb{D}_w$ is equipped with
the hyperbolic metric.}
\item{An orientation-preserving minimal Lagrangian diffeomorphism $F: \mathbb{D}_{z_-} \to \mathbb{D}_{z_+}$ where $\mathbb{D}_{z_\pm}$ are equipped with
the hyperbolic metric.}
\item{A quasisymmetric homeomorphism $h:\mathbb{S}^1\to \mathbb{S}^1$.}
\end{enumerate}
\end{proposition}

Here are some remarks regarding Proposition 3.4:

\begin{itemize}
    \item \textbf{(1) $\to$ (2):} The orientation-preserving minimal Lagrangian diffeomorphism \( G: \mathbb{D}_w \to \mathbb{D}_w \) is obtained by defining  
    \[
    G = G_+ \circ G_-^{-1}
    \]
    where $(G_+, G_-) : \mathbb{D}_z \to \mathbb{D}_w \times \mathbb{D}_w $ is the Gauss map of the maximal disc $ \sigma:\mathbb{D}_z \to \mathbb{A}\mathrm{d}\mathbb{S}^{2,1}$ (see Theorem 4.1 of \cite{AAW00} for more details).

    \item\textbf{(2) $\to$ (3):} The harmonic maps $F_\pm$ can be obtained
    from $G_\pm$ as previously explained (see Lemma 2.1 of \cite{AAW00}).

    \item \textbf{(3) $\to$ (4):} The quasisymmetric homeomorphism \( h: \mathbb{S}^1 \to \mathbb{S}^1 \) is obtained by restricting \( F \) to the boundary \( \partial \mathbb{D} \), i.e.,  
    \[
    h := F|_{\partial \mathbb{D}}=G|_{\partial\mathbb{D}}.
    \]
Here the identification of $\partial{\mathbb{D}_{z_-}}$ with $\partial{\mathbb{D}}$
is made via the restriction of the coordinate map of $\mathbb{D}_{z_-}$.
    \item \textbf{(4) $\to$ (1):} Given a quasisymmetric homeomorphism \( h: \mathbb{S}^1 \to \mathbb{S}^1 \), the corresponding maximal disc \( \sigma:\mathbb{D}_z \to \mathbb{A}\mathrm{d}\mathbb{S}^{2,1} \) is constructed such that its boundary \( \partial \Sigma =\partial (\sigma(\mathbb{D}_z)) \) is the graph of \( h \)
(see Theorem 1.10 of \cite{BS10} for more details).

\end{itemize}

%\subsection{Maximal discs of Weil-Petersson class}

By Remark 5.12 of \cite{BS18},
for an orientation preserving minimal Lagrangian diffeomorphism $G:\mathbb{D}_z\to \mathbb{D}_z$ and the hyperbolic metric $g_z$ determining the conformal structure of $\mathbb{D}_z$,
there exists a $g_z$-self-adjoint  endomorphism $b\in\mathrm{End}(T\mathbb{D})$ satisfying the following conditions:
\begin{equation}\label{e:b-def}
G^* g_z = g_z(b\cdot, b\cdot), \qquad d^\nabla b=0, \qquad \mathrm{det}\, b=1
\end{equation}
where $\nabla$ denotes the Levi-Civita connection of $g_z$. Now, we obtain

\begin{proposition}\label{p:h-determines-I}
Let $G: \mathbb{D}_z \to \mathbb{D}_z$  be an orientation-preserving diffeomorphism satisfying the condition \eqref{e:b-def}.
Consider the map $\sigma_{G,b}:\mathbb{D}_z\to \mathrm{Isom}(\mathbb{D}_z)$ defined 
such that for  $x\in\mathbb{D}$, $\sigma_{G,b}(z(x))$ is the unique isometry 
$$\gamma\in \mathrm{Isom}(\mathbb{D}_z)\cong \mathrm{PSL}(2,\mathbb{R}) \cong \mathbb{A}\mathrm{d}\mathbb{S}^{2,1} $$  satisfying the following conditions: 
\begin{equation}\label{e:reflection}
\gamma \circ G  (z(x)) = z(x), \qquad d\gamma_{G(z(x))}\circ dG_{z(x)}=-  b_{z(x)}.
\end{equation}
Then, the following equality holds:
\begin{equation}\label{e:1/4-equality}
\sigma_{G,b}^*( g_{\mathbb{A}\mathrm{d}\mathbb{S}^{2,1}}) =\frac14 g_z ((E+b)\cdot, (E+b)\cdot )   
\end{equation}
where $E$ denotes the identity on $\mathrm{End}(T\mathbb{D})$.
\end{proposition}

\begin{proof}
The claim for the case of $\mathbb{D}_w$ is established in Proposition 5.5 of \cite{BS18}. 
Now, for a given diffeomorphism $G:\mathbb{D}_z\to\mathbb{D}_z$,  consider the following
commutative diagram:
\begin{equation}\label{e:CD-D-z} 
\begin{tikzcd}
\large
 & \mathbb{D}_{w}   \arrow[d, swap , "z"]  \arrow[r, "\widetilde{G}"]   & \mathbb{D}_w  \arrow[d,   "z"]  \\
& \mathbb{D}_z \arrow [r, "G"] & \mathbb{D}_z 
\end{tikzcd} 
\end{equation}
where $\widetilde{G}$ and $G$
are diffeomorphisms over $\mathbb{D}_w$ and $\mathbb{D}_z$ respectively, satisfying the given conditions in \eqref{e:b-def}. 
By assumption, the diffeomorphism ${G}:\mathbb{D}_z\to\mathbb{D}_z$ also satisfies the conditions in \eqref{e:reflection}.  In particular,  for  $x\in\mathbb{D}$ ,  the map ${\sigma}_{{G},{b}}$  maps $(z(x))$ to an isometry ${\gamma}\in \mathrm{Isom}(\mathbb{D}_z)\cong \mathrm{PSL}(2,\mathbb{R})$  that satisfies the equalities in \eqref{e:reflection}. 
From \eqref{e:CD-D-z},  note that the identification between $\mathrm{Isom}(\mathbb{D}_z)$ and $\mathrm{Isom}(\mathbb{D}_w)$  is given by 
$$
\gamma \in \mathrm{Isom}(\mathbb{D}_z)  \to \tilde{\gamma}:=z^{-1}\circ {\gamma}\circ  z\in \mathrm{Isom}(\mathbb{D}_w).
$$
Since ${\sigma}_{{G},{b}}(z(x))\in \mathrm{Isom}(\mathbb{D}_z) \cong \mathbb{A}\mathrm{d}\mathbb{S}^{2,1}$
for $x\in\mathbb{D}_w$, it follows  that
$$
\tilde{\gamma}:= \tilde{\sigma}_{\widetilde{G},\tilde{b}}(x) =  z^{-1}\circ {\sigma}_{{G},{b}}(z(x)) \circ z = z^{-1}\circ {\gamma}\circ z.
$$
We define $\tilde{b}\in\mathrm{End}(T\mathbb{D})$ by
\begin{equation}\label{e:new-b}
\tilde{b}= dz^{-1}\circ  {b}\circ dz,
\end{equation}
which satisfies \eqref{e:b-def} and \eqref{e:reflection} for Levi-Civita connection of $g_w$.
Here, the equality $d^{\widetilde{\nabla}}  \tilde{b}=0$ follows from the corresponding equality for ${b}$ and the relation
$$
\widetilde{\nabla} = dz^{-1} \circ {\nabla} \circ dz^{-1},
$$ where ${\nabla}$ is the Levi-Civita connection of $g_z$.

Applying Proposition 5.5 of \cite{BS18} under the above conditions for the special case $\mathbb{D}_w$, 
the following equality holds
\begin{equation}\label{e:1/4-equality-w}
\tilde{\sigma}_{\widetilde{G},\tilde{b}}^*( g_{\mathbb{A}\mathrm{d}\mathbb{S}^{2,1}}) =\frac14 g_w ((E+\tilde{b})\cdot, (E+\tilde{b})\cdot ). 
\end{equation}
Then, finally, we obtain
\begin{multline*}
\sigma_{{G}, {b}}^*( g_{\mathbb{A}\mathrm{d}\mathbb{S}^{2,1}}) = (z^{-1})^* (\tilde{\sigma}_{\widetilde{G},\tilde{b}})^*( g_{\mathbb{A}\mathrm{d}\mathbb{S}^{2,1}}) 
= \frac14 (z^{-1})^* \big(g_w ((E+\tilde{b})\cdot, (E+\tilde{b})\cdot )\big)\\
=\frac14 g_w  ((E+\tilde{b}) dz^{-1}\cdot, (E+\tilde{b}) dz^{-1}\cdot )
 = \frac14 g_w  ( dz^{-1}(E+ dz \,\tilde{b}\, dz^{-1})\cdot, dz^{-1}(E+ dz\, \tilde{b}\, dz^{-1})\cdot )\\
 = \frac14 (z^{-1})^*  g_w  ( (E+ dz \,\tilde{b}\, dz^{-1})\cdot, (E+ dz \,\tilde{b}\, dz^{-1})\cdot )
= \frac14 (g_z ( (E+  {b} )\cdot, (E+  {b})\cdot ).
\end{multline*}
This completes the proof.
\end{proof}

 Hence, by Propositions \ref{p:equiv-rel} and  \ref{p:h-determines-I},
the first fundamental form $I$ over the maximal conformal embedding ${\sigma}:\mathbb{D}_w\to\mathbb{A}\mathrm{d}\mathbb{S}^{2,1}$ is given by
\begin{equation}
I =\frac14 g_w ((E+b)\cdot, (E+b)\cdot ),  
\end{equation}
where $b\in \mathrm{End}(T\mathbb{D})$ satisfies conditions \eqref{e:b-def}
and \eqref{e:reflection} for $g_w$.
\begin{remark}\label{r:zero-b}
We apply Proposition \ref{p:h-determines-I} in the special case where  $b=0$. This condition implies
that $G$ is the identity map on $\mathbb{D}_z$. Hence, the equalities in \eqref{e:reflection} indicates that $\gamma\in\mathrm{PSL}(2,\mathbb{R})$ represents an
involutional rotation by an angle $\pi$ around the point $z(x)$, which we denote by $\mathcal{I}_{z(x)}$.  Therefore, the image of
$\sigma_{G,b}:\mathbb{D}_z\to \mathbb{A}\mathrm{d}\mathbb{S}^{2,1}$ is 
the totally geodesic disc
\begin{equation}
\mathcal{R}_{\pi}:=\{\, \mathcal{I}_{z(x)}: x \in\mathbb{D}\, \} \subset \mathbb{A}\mathrm{d}\mathbb{S}^{2,1},
\end{equation}
which can be identified with $\mathbb{D}_w$.  In this setting, the map $\sigma_{G,b}:\mathbb{D}_z\to \mathcal{R}_{\pi}\cong\mathbb{D}_w$ coincides the map
$w_\mu$ where $\mu$ is the Beltrami differential of $z$.  Recall that $w_\mu:\mathbb{D}_z\to\mathbb{D}_w$ is a conformal map. Additionally, in this case,
we observe that $G_+=G_-$ is the identity map on $\mathbb{D}_w$, leading to the identification $\mathbb{D}_{z_+}
=\mathbb{D}_{z_-}=\mathbb{D}_z$.
\end{remark}

%\begin{remark}
%Recall that the universal Teichm\"uller space $T(1)$ admits a natural $\mathrm{PSL}(2,\mathbb{R})$-symmetry, where
%$\gamma\in\mathrm{PSL}(2,\mathbb{R})$ acts on a quasisymmetric homeomorphism $h:\mathbb{S}^1\to \mathbb{S}^1$ by 
%$$\gamma\cdot h=\gamma\circ h. 
%$$
%By Proposition \ref{p:equiv-rel}, each quasisymmetric homeomorphism $h$ uniquely determines a maximal conformal %embedding $\sigma:\mathbb{D}_z\to  \mathbb{A}\mathrm{d}\mathbb{S}^{2,1}$. From the proof of Proposition \ref{p:h-%determines-I},  we observe that
%the action of $\gamma $ on $h$ corresponds to the action $\gamma$ on the coordinate map $z$, given by
%$$
%\gamma\cdot z=\gamma\circ z.
%$$
%Thus,  it  follows that $\gamma$ acts on the maximal disc $\Sigma=\sigma(\mathbb{D}_z)$ by 
%$$
%\gamma\cdot \delta  = \gamma \circ \delta \circ \gamma^{-1}, \qquad \text{for} \quad
%\delta\in \Sigma.
%$$
%\end{remark}

\begin{definition} Let ${\sigma}:\mathbb{D}_z\to \mathbb{A}\mathrm{d}\mathbb{S}^{2,1}$ be a maximal
conformal embedding such that the boundary $\partial\Sigma=\partial (\sigma(\mathbb{D}_z))$ is the graph of a quasisymmetric homeomorphism $h$ representing an element in $T(1)$.
We say that the maximal conformal embedding ${\sigma}:\mathbb{D}_z\to \mathbb{A}\mathrm{d}\mathbb{S}^{2,1}$ is of \emph{ Weil-Petersson class} if 
the corresponding quasisymmetric homeomorphism $h$ represents an element in $T_0(1)$.
\end{definition}

\begin{proposition}\label{t:total-curv} For a maximal conformal embedding of Weil-Petersson class $$\sigma:\mathbb{D}_z \to \mathbb{A}\mathrm{d}\mathbb{S}^{2,1}
$$
with the induced metric $e^{\phi}|dz|^2$ and
the induced Gauss map $(F_+,F_-):\mathbb{D}_z\to\mathbb{D}_{z_+}\times \mathbb{D}_{z_-}$, we have:
\begin{enumerate}
\item The total curvature is finite, that is,
\begin{equation*}
\int_{\mathbb{D}} \kappa^2 e^{\phi}\ d^2z= \int_{\mathbb{D}} |\mu_{G_\pm}|^2 e^\phi\, d^2z = \int_{\mathbb{D}} |\mu_{F_\pm}|^2 e^\phi\, d^2z < \infty.
\end{equation*}
\item The anti-holomorphic energies of $G_\pm$ and $F_\pm$  are finite respectively, that is,
\begin{equation*}
\int_{\mathbb{D}} |\Phi(G_\pm)|^2 e^{-\phi}\, d^2z= \int_{\mathbb{D}} |\Phi(F_\pm)|^2 e^{-\phi}\, d^2z < \infty.
\end{equation*}
\end{enumerate}
\end{proposition}

\begin{proof}
Note that the quasisymmetric homeomorphism $h:\mathbb{S}^1\to \mathbb{S}^1$ is given by the restriction of $F:=F_+\circ F_-^{-1}$ to
the boundary $\partial\mathbb{D}$.
Then, by the condition that $\sigma:\mathbb{D}_z \to \mathbb{A}\mathrm{d}\mathbb{S}^{2,1}$ is of Weil-Petersson class and using Theorem B of \cite{DMSST26}, we have the following:
\begin{equation}
\int_{\mathbb{D}} |\mu_{F}|^2 e^{\psi}\, d^2z_- < \infty.
\end{equation}
Now,  form \eqref{e:composition}, the Beltrami differential $\mu_F$ can be expressed as:
\begin{equation}
\mu_F \circ F_- \frac{ \overline{(F_-)_{z}}}{ (F_-)_{z}} = \frac {\mu_+- \mu_-}{1- \mu_+ \overline{\mu}_-} = \frac{2\mu_+}{1+|\mu_+|^2},
\end{equation}
where $\mu_\pm=\mu_{F_\pm}$. This leads to the expression:
\begin{equation}\label{e:mu12}
\left| \mu_F\circ F_-  \right|= \frac{2|\mu_+|}{1+|\mu_+|^2}=  \frac{2|\mu_-|}{1+|\mu_-|^2}.
\end{equation}
Furthermore, we have
\begin{multline}\label{e:est1}
\int_{\mathbb{D}} |\mu_F|^2 e^{\psi} \, d^2z_-
=\, \int_{\mathbb{D}} \left( |\mu_F|^2 \circ F_-\right)\, e^{\psi \circ F_-} \, \left(  \left| (F_-)_z \right|^2 - \left|(F_-)_{\bar{z}}\right|^2 \right) \, d^2z\\
=\, \int_{\mathbb{D}} \left(|\mu_F|^2 \circ F_-\right)\, e^{\psi \circ F_-} \,   \left| (F_-)_z\right|^2 \left(1 - \left|\mu_-\right|^2 \right) \, d^2z\\
=\, \int_{\mathbb{D}} \frac{4|\mu_-|^2}{(1+|\mu_-|^2)^2} e^{\phi} \left( 1 - \left|\mu_- \right|^2 \right) \, d^2z.
\end{multline}
Since  $[\mu_F]\in T_0(1)$,  we have $|\mu_F(z)|\to 0$ as $|z|\to 1$, implying that $|\mu_-(z)|\to 0$ as $z$ approaches to $\partial\mathbb{D}$ by \eqref{e:mu12}.
Hence, there exists $C>0$ such that
\begin{equation}\label{e:lower-bound}
C< \frac{1 -|\mu_-|^2}{(1+|\mu_-|^2)^2} < 1 \qquad \text{over} \quad \mathbb{D}.
\end{equation}
Hence, from \eqref{e:est1} and \eqref{e:lower-bound}, we conclude that
\begin{multline}\label{e:mu-Phi-F}
\int_{\mathbb{D}} |\mu_F|^2 e^{\psi}\, d^2z_- <\infty \qquad \text{ if and only if} \qquad
\int_{\mathbb{D}} {|\mu_\pm|^2} e^{\psi} \, d^2z \leq \int_{\mathbb{D}} {|\mu_\pm|^2} e^{\phi} \, d^2z < \infty.
\end{multline}
This means that $F_\pm|_{\partial{\mathbb{D}}}$ represents points in $T_0(1)$, which
also implies that $|\mu_{F_\pm}|(z) \to 0$ as $|z|\to 1$. Now, recall that there exists a nonzero constant $a$ such that the curvature satisfies
$$-1< K_\phi < -a^2$$ by \eqref{e:Gaussian-curv} and $|\mu_\pm|\to 0$. By Lemma 4.9 of \cite{Se16}, it follows that
$$ 
e^\phi < a^{-2} e^\psi.
$$ Using these facts and the identity $\kappa^2=|\mu_{F_\pm}|^2$, we obtain
\begin{equation}
\int_{\mathbb{D}} \kappa^2 e^{\phi}\ d^2z= \int_{\mathbb{D}} |\mu_{F_\pm}|^2 e^\phi\, d^2z <a^{-2} \int_{\mathbb{D}} |\mu_{F_\pm}|^2 e^\psi\, d^2z < \infty.
\end{equation}
This completes the proof of item (1). The proof of item (2) follows easily by noting that
$$
|\Phi(F_{\pm})|^2\, e^{-\phi} = |\mu_{F_\pm}|^2 \, e^{\phi}.
$$
\end{proof}

\subsection{Mess map}

%By Proposition \ref{p:equiv-rel}, the following correspond to the same element in the universal %Teichm\"uller space $T(1)$:
%\vspace{7pt}
%\begin{enumerate}
%\setlength{\itemsep}{7pt}
%\item{The conformal structure of $\mathbb{D}_{z_\pm}$,}
%\item{The quasisymmetric homeomorphism $z_\pm|_{\mathbb{S}^1}= (F_\pm\circ z)|_{\mathbb{S}^1}$,}
%\item{The Beltrami differential $\mu_{z_\pm}$ associated with the map $z_\pm$.}
%\end{enumerate}

For a given pair $(\mu, \Phi)$ representing a point in the holomorphic cotangent bundle $T^*T(1)$,  we first take a conformal structure over the unit disc
determined by $\mu$, which we denote by $\mathbb{D}_z$.
%We then define  
%$$
%I(\mu,\Phi):=e^\phi |dz|^2
%$$ 
%where $\phi$ denotes the solution of the Gauss equation \eqref{e:Gauss} with $\Phi$ over $\mathbb{D}_z$.
For a given the holomorphic quadratic differential $\Phi$ over $\mathbb{D}_z$ and $\mathbb{D}_w$,
by Theorem 3.2 of \cite{TW94}, there exist harmonic maps 
$$G_\pm:\mathbb{D}_z\to \mathbb{D}_w$$ 
with the Hopf differential $\pm\Phi$ respectively.
Then, the pullback by $G_\pm$ of the hyperbolic metric on $\mathbb{D}_w$ are two different hyperbolic metrics 
on $\mathbb{D}_z$, denoted by $I_\pm$,  which determines two conformal structures $\mathbb{D}_{z_\pm}$ respectively.
 As explained previously, there exist harmonic maps 
\begin{equation*}
F_\pm: \mathbb{D}_z \to \mathbb{D}_{z_\pm}
\end{equation*}
with the Hopf differentials $\pm\Phi$ respectively.
Finally, we introduce the Mess map defined as
\begin{equation}\label{e:def-Mess}
\mathrm{Mess}: T^*T(1) \to T(1)\times T(1),
\end{equation}
which maps a pair $(\mu,\Phi)$ to the pair $(I_+,I_-)$ representing a point in $T(1)\times T(1)$. 

\begin{proposition}\label{p:bijective-Mess}
The map $\mathrm{Mess}: T^*T(1) \to T(1)\times T(1)$ is a bijective map.
\end{proposition}

\begin{proof}
To show the injectivity of $\mathrm{Mess}:T^*T(1) \to T(1)\times T(1)$, we consider the following equalities
\begin{equation}\label{e:composition-mu}
\mu_{z_\pm} = \frac{\mu_z \pm z^*(\mu_+)}{1\pm \bar{\mu}_z\, z^*(\mu_+)}
\end{equation}
where $\mu_+=\mu_{F_+}$. If the $\mathrm{Mess}$ map is not injective, there are two pairs $(a,b)$ and $(a',b')$ for $(\mu_z, z^*(\mu_F))$ that satisfy the same equations 
for a given $\mu_{z_\pm}$, leading to the equalities:
\begin{equation}\label{e:injective}
a\pm b +\bar{a'}bb' \pm a\bar{a'}b' = a' \pm b' +\bar{a}bb' \pm \bar{a}a'b.
\end{equation}
By rearranging and combing these expressions, we obtain
\begin{equation}
a-a'= (\bar{a}-\bar{a'}) bb'.
\end{equation}
From this, if $a\neq a'$, we have $|bb'|=1$. However, this contradicts the fact that $|bb'|<1$. Hence, we conclude that $a=a'$, which further implies that $b=b'$ by the previous equation  \eqref{e:injective}.
Thus, there exists a unique pair $(\mu_z, \mu_{F_+})$ satisfying \eqref{e:composition-mu} for a given pair
$(\mu_{z_+},\mu_{z_-})$.  Since the pair $(\mu_z,\mu_{F_+})$ uniquely determines
the pair $(\mu_z, \Phi)$  by \eqref{e:quad-Beltrami}, this proves the injectivity of the map $\mathrm{Mess}:T^*T(1) \to T(1)\times T(1)$.

To prove the surjectivity of the map $\mathrm{Mess}:T^*T(1) \to T(1)\times T(1)$,  let $(I_+,I_-)$ be a pair representing a point in $T(1)\times T(1)$
be a given pair.  We first consider two conformal structures over the unit disc determined by $I_+,I_-$,
denote by $\mathbb{D}_{z_+}$ and $\mathbb{D}_{z_-}$,  respectively. This implies that there exist two quasiconformal maps $z_{\pm}:\mathbb{D}_w\to\mathbb{D}_{z_\pm}$.
Next, we consider the conformal structure on the unit disc $\mathbb{D}$ determined by $I_++I_-$, which we denote by $\mathbb{D}_z$.
By applying Proposition \ref{p:equiv-rel}, for the fixed $\mathbb{D}_z$ and the quasisymmetric homeomorphism
$(z_+\circ z_-^{-1})|_{\partial\mathbb{D}}$, there exists a maximal conformal embedding $\sigma:\mathbb{D}_z\to  \mathbb{A}\mathrm{d}\mathbb{S}^{2,1}$ with its induced Gauss map $(F_+,F_-):\mathbb{D}_z\to \mathbb{D}_{z_+}\times \mathbb{D}_{z_-}$. By the construction, the pair $(\mu_z, \Phi(F_+))$ is mapped  to
the given pair $(I_+,I_-)$ by the map $\mathrm{Mess}$. This completes the proof of the surjectivity
of $\mathrm{Mess}$.

\end{proof}

%\begin{remark}\label{r:equiv-rel}
%By the above constructions of the induced Gauss map $F_\pm$, one can see that 
%the one-to-one correspondence of the items (3) and (4) in Proposition \ref{p:equiv-rel} still holds in this general %situation.
%\end{remark}

%\begin{remark}\label{r:zero-b-diagonal}
%By Remark \ref{r:zero-b},  it is evident that the  $\mathrm{Mess}$ map sends the image of zero section
%of $T^*T(1)$ to the diagonal set in $T(1)\times T(1)$.
%\end{remark}

When restricting the Mess map to the holomorphic cotangent bundle of the Weil-Petersson Teichm\"uller space $T_0(1)$, we have

\begin{theorem}\label{t:Mess-diff}
The following restriction of  the $\mathrm{Mess}$ map is a diffeomorphism,
\begin{equation}
\mathrm{Mess}:T^*T_0(1) \to T_0(1)\times T_0(1).
\end{equation}
\end{theorem}

\begin{proof}
Recall that a point in $T^*T_0(1)$ is represented by a pair $(\mu, \Phi)$,  where $\mu$ is a Beltrami differential on $\mathbb{D}_w$ and $\Phi$ is
a holomorphic quadratic differential on $\mathbb{D}_z$ satisfying the conditions
\begin{equation}\label{e:condition-WP}
\int_{\mathbb{D}} |\mu|^2 e^{\psi} \, d^2w <\infty, \qquad \int_{\mathbb{D}} |\Phi|^2 e^{-\psi} \, d^2z < \infty.
\end{equation}
Here $\mathbb{D}_w$ denotes the origin in $T_0(1)$ and $\mathbb{D}_z$ represents the conformal structure determined by $\mu$.
The first inequality follows from Lemma 3.3 of \cite{TT06}, and the second inequality follows from the definition of $T^*T_0(1)$.
These conditions imply that the restrictions
of maps $F_+,F_-$ and $z$  to the boundary are quasisymmetric homeomorphisms of $\mathbb{S}^1$ representing points in $T_0(1)$ respectively.
Since the composition of two quasisymmetric homeomorphisms representing elements in $T_0(1)$ is
itself  a quasisymmetric homeomorphism in $T_0(1)$,
the restriction of $z_\pm=F_\pm\circ z$ to the boundary $\mathbb{D}$ is a quasisymmetric homeomorphism in $T_0(1)$. Hence, it follows that $\mathrm{Mess}$ maps  $T^*T_0(1)$ 
into $T_0(1)\times T_0(1)$.  

For given quasisymmetric homeomorphisms $f_\pm:\mathbb{S}^1\to\mathbb{S}^1$ representing two points in $T_0(1)$,
let us denote by $z_\pm:\mathbb{D}_w\to \mathbb{D}_{z_\pm}$ the corresponding conformal structures equipped with
the hyperbolic structures $I_\pm$ respectively.
By Proposition \ref{p:bijective-Mess},
there exists a pair $(\mu_z,\Phi)$ representing a point in $T^*T(1)$ which is mapped to $(I_+,I_-)$ by the $\mathrm{Mess}$ map. 
We need to show that $(\mu_z,\Phi)$  represents a point in $T^*T_0(1)$. To prove this, let us consider the quasisymmetric homeomorphism $f_+\circ f_-^{-1}:\mathbb{S}^1\to\mathbb{S}^{1}$. Then, by Proposition \ref{p:equiv-rel}, 
there exists a minimal Lagrangian diffeomorphism extension $F:\mathbb{D}_{z_-}\to \mathbb{D}_{z_+}$ of $f_+\circ f_-^{-1}$. Moreover,  
there exist harmonic maps $F_\pm: \mathbb{D}_z \to \mathbb{D}_{z_\pm}$ having $\pm\Phi$ as the Hopf differentials, respectively. By the construction, 
we have $F=F_+\circ F_-^{-1}$ , and it follows that
\begin{equation}
\int_{\mathbb{D}} |\mu_F|^2 e^{\psi} \, d^2z_-  < \infty
\end{equation}
by Lemma 3.3 in \cite{TT06}. From this, as in the proof of Proposition \ref{t:total-curv}
we have the equivalence
\begin{multline}\label{e:mu-Phi-F}
\int_{\mathbb{D}} |\mu_F|^2 e^{\psi}\, d^2z_- <\infty \qquad \text{ if and only if} \qquad
\int_{\mathbb{D}} {|\mu_\pm|^2} e^{\psi} \, d^2z \leq \int_{\mathbb{D}} {|\mu_\pm|^2} e^{\phi} \, d^2z < \infty.
\end{multline}
This means that $F_\pm|_{\partial{\mathbb{D}}}$ represents points in $T_0(1)$.
For the given map $f_\pm=z_\pm|_{\partial{\mathbb{D}}}$,  which represent points in $T_0(1)$ respectively, we have that $z|_{\partial{\mathbb{D}}}=( F_\pm^{-1}\circ z_\pm ) |_{\partial{\mathbb{D}}}$ represents a point in $T_0(1)$. 
From  the condition $\int_{\mathbb{D}} {|\mu_\pm|^2} e^{\phi} \, d^2z < \infty$ in \eqref{e:mu-Phi-F}, it follows that $\Phi\in T^*T_0(1)$.
Hence, the pair $(\mu_z, \Phi)$ represents a point in $T^*T_0(1)$. This concludes that the map $\mathrm{Mess}:T^*T_0(1) \to T_0(1)\times T_0(1)$ is surjective.

The injectivity of $\mathrm{Mess}:T^*T_0(1) \to T_0(1)\times T_0(1)$ can be proved similarly to the proof of Proposition \ref{p:bijective-Mess}.

To complete the proof of the claim, it suffices to show that the differential of 
 the map $\mathrm{Mess}:T^*T_0(1) \to T_0(1)\times T_0(1)$ 
is an isomorphism. Then, by the inverse function theorem, the proof will be completed. This will be established in Proposition \ref{p:diffeo}.
\end{proof}

As in the proofs of Proposition \ref{t:total-curv} and Theorem \ref{t:Mess-diff}, one can similarly establish the following result.

\begin{theorem}\label{t:energy-finite}
For the induced harmonic maps $F_\pm$ associated to the pairs $(\mu,\Phi)$ representing points in $T^*T_0(1)$, the anti-holomorphic energy of $F_\pm$ defines a finite valued functional $E$ over $T^*T_0(1)$.
\end{theorem}

Now we consider the following commutative diagram:

\begin{equation}\label{e:CD-WP-fibration} 
\begin{tikzcd} 
\large
   T^*T_0(1)  \arrow{rr}{\mathrm{Mess}}  \arrow[swap]{dr}{p_1}\ \  &    &   \ \  T_0(1)\times T_0(1)  \arrow{dl}{p_2} \\
& T_0(1) &
\end{tikzcd} 
\end{equation}
Here the projection $p_1$ 
is defined as follows: for a given pair $(\mu, \Phi)$  representing a point in $T^*T_0(1)$, there exist harmonic maps $G_\pm:\mathbb{D}_z\to \mathbb{D}_w$  whose Hopf differentials are $\pm\Phi$ respectively, by Theorem 3.2 of \cite{TW94}.
Then, $p_1$ maps the pair $(\mu, \Phi)$ to the quasisymmetric homeomorphism $h:= (G_+\circ G_-^{-1})|_{\partial\mathbb{D}}$, which represents an element in $T_0(1)$ by the proof of Theorem \ref{t:Mess-diff}. 
The second projection $p_2$ is defined by
\begin{equation}\label{e:mu-composition}
p_2(\mu_+,\mu_-) = [\mu_+\star \mu_-^{-1}]
\end{equation}
where the operation $\star$ is defined in terms of
\eqref{e:composition}. Then, we observe that the above diagram is commutative, that is,
$p_1= \mathrm{Mess}\circ p_2$. This follows from the constructions of these maps and
\[
G_+\circ G_-^{-1}|_{\partial\mathbb{D}}= F_+\circ F_-^{-1}|_{\partial\mathbb{D}} = z_+\circ z_-^{-1} |_{\partial\mathbb{D}}.
\]

\begin{remark}\label{r:totally-geodesic}
The Weil-Petersson universal Teichm\"uller space $T_0(1)$ parametrizes the
space of the maximal discs of Weil-Petersson class by Theorem 
1.10 of \cite{BS10}.  By the definition of maximal conformal embedding
$\sigma: \mathbb{D}_z \to \mathbb{A}\mathrm{d}\mathbb{S}^{2,1}$, the cotangent bundle
$T^*T_0(1)$ parametrizes the space of maximal conformal embeddings of Weil-Petersson class. Using \eqref{e:CD-D-z} and \eqref{e:mu-composition}, we observe that
$T_0(1)$-copy of maximal conformal embeddings corresponds to the same maximal disc. In particular, by Remark \ref{r:zero-b}, the inverse images by $p_i$ for $i=1,2$ of a totally geodesic disc correspond to the image of the zero section of $T^*T_0(1)$ and the diagonal set in $T_0(1)\times T_0(1)$, respectively.
\end{remark}

\begin{remark}\label{r:parametrization}
Recall that $T_0(1)^\pm$ is the inverse image under $\mathrm{Mess}$ of the subsets $T_0(1)\times \{0\}$ and $\{0\} \times T_0(1)$ in $T_0(1)\times T_0(1)$. 
Using \eqref{e:CD-D-z} and \eqref{e:mu-composition}, we observe that these spaces $T_0(1)^\pm$ also parametrize the space of maximal discs of Weil-Petersson class in $\mathbb{A}\mathrm{d}\mathbb{S}^{2,1}$, respectively. 
\end{remark}

\section{Symplectic structure on $T^*T_0(1)$}\label{s:variation}

In this section, we first derive some variational formulas for several quantities associated with a maximal disc in $\mathbb{A}\mathrm{d}\mathbb{S}^{2,1}$, considered along the deformations of a maximal conformal embedding. Using these formulas, we establish a relationship between  the canonical form of $T^*_0T(1)$ and the difference of  pullbacks of Weil-Petersson symplectic forms from each factor of $T_0(1)\times T_0(1)$ via the $\mathrm{Mess}$ map.

\subsection{Variational formulas}
For a maximal conformal  embedding $\sigma:\mathbb{D}_z \to \mathbb{A}\mathrm{d}\mathbb{S}^{2,1}$, we denote by 
\begin{equation}\label{e:maximal-conformal-embedding}
\sigma^\epsilon:\mathbb{D}_{z^\epsilon}\to \mathbb{A}\mathrm{d}\mathbb{S}^{2,1} 
\end{equation}
its deformation family for
a small real parameter $\epsilon$. In general, such a deformation consists of two parts: one is a deformation of conformal structures on the domain of $\sigma^\epsilon$,
and the other is a deformation of  the maximal conformal embedding into  $\mathbb{A}\mathrm{d}\mathbb{S}^{2,1}$. 
For the deformation of the conformal structure of the domain of $\sigma^\epsilon$, we denote it by
$\mathbb{D}_{z^\epsilon}$. The other part of the deformation of  the maximal conformal embeddings is determined by  deformation of
the induced Gauss map $(F_+,F_-):\mathbb{D}_z \to \mathbb{D}_{z_+}\times \mathbb{D}_{z_-}$.

To address such a general situation involving the deformation of a maximal conformal embedding $\sigma: \mathbb{D}_z \to \mathbb{A}\mathrm{d}\mathbb{S}^{2,1}$, we consider
 the following diagram:
\begin{equation}\label{e:CD-1}
\begin{tikzcd}
\Large
\mathbb{D}_z  \arrow[r, "F"]  \arrow[d, swap, "f^\epsilon"]  \arrow[dr,  "H^\epsilon"]  &  \mathbb{D}_u \arrow[d, "h^\epsilon"] \\
\mathbb{D}_{z^\epsilon} \arrow[r, swap, "{F}^{\epsilon}"]   &  \mathbb{D}_{u^\epsilon}
\end{tikzcd} 
\end{equation}
Here $f^\epsilon$and $h^\epsilon$ denote quasiconformal maps with the corresponding Beltrami differential $\nu^{\epsilon}_f$ and $\nu^{\epsilon}_h$,
respectively.   We may assume that $\nu_f^{\epsilon}$ and $\nu_h^{\epsilon}$ depend  analytically on the real parameter $\epsilon$ such that $\nu^0_f=0$ and $\nu^0_h=0$.  Therefore, the quasi-conformal maps  $f^{\epsilon}$, $h^{\epsilon}$ satisfy the following Beltrami differential equations:
\begin{equation}\label{e:Bel-eq}
 f^\epsilon_{\bar{z}} =  \nu^{\epsilon}_f f^{\epsilon}_{z}, \qquad h^{\epsilon}_{\bar{u}}= \nu^{\epsilon}_h h^{\epsilon}_{u}.
\end{equation}
It has been known that $f^\epsilon$ and $h^\epsilon$ depend analytically on $\epsilon$ for every fixed $z$. 
Taking derivative at $\epsilon=0$, we obtain
\begin{equation}\label{e:der-Bel}
\dot{f}_{\bar{z}}= \nu_f, \qquad \dot{h}_{\bar{u}}=\nu_h
\end{equation}
where $\nu_f$ and $\nu_h$ denote the harmonic Beltrami differential given by the derivative of $\nu_f^{\epsilon}$
and $\nu^\epsilon_h$ at $\epsilon=0$ respectively.

\begin{lemma}\label{l:key}
For $\nu_h=\dot{h}_{\bar{z}}$ satisfying $h^\epsilon \circ F= H^{\epsilon}$, we have
\begin{equation}\label{e:key}
\nu_h=R(\dot{\mu}_H, \mu_F) = \Big(\frac{\dot{\mu}_H}{1-|\mu_F|^2} \frac{F_z}{\overline{F}_{\bar{z}}}\Big)\circ {F}^{-1}.
\end{equation}
Here $\dot{\mu}_H=\frac{d}{d\epsilon}\big|_{\epsilon=0} \mu_{H^\epsilon}$.
\end{lemma}
\begin{proof}
From \eqref{e:composition}, we derive
the following expression for the Beltrami differential $\mu_{H^\epsilon}$ of $H^\epsilon$:
\begin{equation}\label{e:mu-F}
\mu_{H^\epsilon}=\frac{\mu_F+ F^*( \nu^{\epsilon}_h)}{1+\bar{\mu}_F F^*(\nu^{\epsilon}_h)}.
\end{equation}
Taking the derivative of this expression with respect to $\epsilon$ at $\epsilon=0$, we obtain
$$
\dot{\mu}_H=(1-|\mu_F|^2) F^*(\nu_h).
$$
From this, the equality in\eqref{e:key} follows directly.
\end{proof}

For a family of quasiconformal maps $f^\epsilon$ with $f^\epsilon_{\bar{z}} =  \nu^{\epsilon} f^{\epsilon}_{z}$ and a smooth family of tensors $\omega^\epsilon$ of type $(\ell,m)$, set
\[
(f^\epsilon)^*(w^\epsilon)= \omega^\epsilon\circ f^\epsilon \big((f^\epsilon)_z\big)^{\ell} \big( (\overline{f^\epsilon})_{\bar{z}}\big)^m.
\]
The Lie derivatives of the family $\omega^\epsilon$ along a vector field $\frac{\partial}{\partial \epsilon_\nu}$ is defined by
\[
L_\nu \omega= \frac{\partial}{\partial \epsilon}\Big|_{\epsilon=0} (f^\epsilon)^*(w^\epsilon) .
\]

\begin{proposition}\label{p:var-nu-F}
If the diagram \eqref{e:CD-1} is commutative, that is, $F^\epsilon\circ f^\epsilon = H^\epsilon= h^\epsilon\circ F$, then
\begin{equation}\label{e:var-nu-F}
L_{\nu_f} \mu_{F} =  (1-|\mu_F|^2) F^*(\nu_h) - (\nu_f-\bar{\nu}_f\mu_F^2).
\end{equation}
\end{proposition}

\begin{proof}
From the commutative diagram in \eqref{e:CD-1}, we have the following relation:
\begin{equation}\label{e:Belt-rel}
\frac{\mu_F +  F^*(\nu^{\epsilon}_h)}{1+ \bar{\mu}_F F^*(\nu^{\epsilon}_h)} =\frac{ \nu^{\epsilon}_f 
+(f^\epsilon)^*(\mu_F^\epsilon)}{1+\bar{\nu}^{\epsilon}_f (f^\epsilon)^*(\mu_F^\epsilon)}.
\end{equation}
Taking the derivative $\epsilon=0$ of this equation, we obtain:
\begin{equation}\label{e:var-nu-F1}
(1-|\mu_F|^2)F^*(\nu_h)= (\nu_f+L_{\nu_f}\mu_F) -\bar{\nu}_f\mu_F^2.
\end{equation}
This completes the proof.
\end{proof}

%\begin{proposition}\label{p:infini-trivial}
%If the diagram \eqref{e:CD-1} is commutative, then
%\begin{equation}\label{e:nu-H}
%\dot{\mu}_H = \nu_f -\bar{\nu}_f \mu_F^2 + L_{\nu_f}\mu_F.
%\end{equation}
%\end{proposition}

%\begin{proof}
%The equality \eqref{e:nu-H}  follows directly from \eqref{e:key} and \eqref{e:var-nu-F}. 
%\end{proof}

\begin{proposition}\label{p:var-Hopf}
For a family of Hopf differentials $\Phi^\epsilon = e^{\psi^\epsilon\circ F^\epsilon}F^\epsilon_z \overline{F}^\epsilon_z$ on 
$\mathbb{D}_{z^\epsilon}$ satisfying the commutative diagram \eqref{e:CD-1}, the Lie derivative of $\Phi$ along the direction of $\nu_f$ is given by
\begin{equation}\label{e:var-Hopf}
\begin{split}
L_{{\nu}_f}\, \Phi =&\, \Phi \big(F^*({{\nu_h}})\bar{\mu}_F-{\bar{\nu}_f}\mu_F\big) +\overline{\Phi} \big( F^*(\bar{{\nu}}_h)\mu_F^{-1}-\bar{{\nu}}_f\mu_F^{-1} \big)\\
=&\, e^\phi  \big(F^*({{\nu_h}})\bar{\mu}_F^2-{\bar{\nu}_f}|\mu_F|^2\big) + e^\phi \big( F^*(\bar{{\nu}}_h)-\bar{{\nu}}_f \big) .
\end{split}
\end{equation}
\end{proposition}

\begin{proof}
From the definition of the Lie derivative, we have
\begin{equation}\label{e:var-Hopf-1}
L_{{\nu}_f}\, \Phi =\frac{\partial}{\partial \epsilon}\Big|_{\epsilon=0} \Big( e^{\psi^\epsilon\circ F^\epsilon\circ f^\epsilon} F^\epsilon_z\circ f^\epsilon\, \overline{F}_{{z}}^{\epsilon}\circ f^{\epsilon}\, f^{\epsilon}_z {f}^\epsilon_{{z}} \Big).
\end{equation}
From the relation $h^\epsilon \circ F= F^{\epsilon}\circ f^{\epsilon}$, we obtain
\begin{equation}\label{e:var-Hopf-2}
(\dot{h}\circ F)_z= \frac{\partial}{\partial\epsilon}\Big|_{\epsilon=0} \Big(F^\epsilon_z\circ f^{\epsilon} f^\epsilon_z\Big) + F_{\bar{z}} \dot{\bar{f}}_z.
\end{equation}
By combining  these expressions, we derive
\begin{align*}
L_{{\nu_f}}\, \Phi =&\ e^{\psi\circ F} \Big( \dot{\psi}+\psi_u\dot{h}+\psi_{\bar{u}}\dot{\bar{h}}\Big)\circ F F_z \overline{F}_{{z}} \\
&+ e^{\psi\circ F} \Big( (\dot{h}\circ F)_z - F_{\bar{z}} \dot{\bar{f}}_z\Big)\overline{F}_{{z}}
+  e^{\psi\circ F} F_z \Big( (\dot{\bar{h}}\circ F)_{\bar{z}} - \overline{F}_{\bar{z}} \dot{\bar{f}}_{{z}}\Big)\\
=&\ e^{\psi\circ F} \Big(\dot{\psi}+\psi_u\dot{h}+\psi_{\bar{u}}\dot{\bar{h}}+ \dot{h}_z + \dot{\bar{h}}_{\bar{z}} \Big)\circ F F_z \overline{F}_{{z}}\\
&+e^{\psi\circ F} \Big( \dot{h}_{\bar{z}}\circ F \overline{F}_z \overline{F}_{{z}} - \dot{\bar{f}}_z  F_{\bar{z}} \overline{F}_{{z}} 
     + \dot{\bar{h}}_{{z}}\circ F {F}_{{z}} {F}_{{z}} - \dot{\bar{f}}_{{z}}  F_{{z}} \overline{F}_{\bar{z}} \Big).
\end{align*}     
Using the following identity for variations of hyperbolic metrics,
\begin{equation}\label{e:Ahl}
\dot{\psi}+\psi_u\dot{h}+\psi_{\bar{u}}\dot{\bar{h}}+ \dot{h}_z + \dot{\bar{h}}_{\bar{z}}=0,
\end{equation}
which follows from the Ahlfors' lemma given in \cite{Ahl61}, we obtain
\begin{align*}
L_{{\nu_f}}\, \Phi=&\ e^{\psi\circ F} F_z\overline{F}_z \Big( \dot{h}_{\bar{z}}\circ F \overline{F}_{\bar{z}}\overline{F}_z (F_z\overline{F}_{\bar{z}})^{-1} -\dot{\bar{f}}_{{z}}F_{\bar{z}}F_z^{-1}\Big)\\ &+
e^{\psi\circ F} F_{\bar{z}}\overline{F}_{\bar{z}} \Big( \dot{\bar{h}}_z\circ F F_z^2 (F_{\bar{z}}\overline{F}_{\bar{z}})^{-1} -\dot{\bar{f}}_z
F_zF_{\bar{z}}^{-1}\Big)\\
=&\  \Phi \big(F^*({{\nu_h}})\bar{\mu}_F-{\bar{\nu}_f}\mu_F\big) +\overline{\Phi} \big( F^*(\bar{{\nu}}_h)\mu_F^{-1}-\bar{{\nu}}_f\mu_F^{-1} \big).
\end{align*}
This completes the proof.
\end{proof}

\begin{remark}
The expressions for the Lie derivatives in  \eqref{e:var-nu-F} and \eqref{e:var-Hopf}  are formulated using both
the harmonic Beltrami differentials on the source disc $\mathbb{D}_z$ and the pullback by $F$ of the harmonic Beltrami differentials on the target disc.
This formulation arises because we consider the general case where deformations occur simultaneously on both the source and target discs of the harmonic map $F$ .  
\end{remark}

We now consider a deformation of the induced harmonic maps associated with
\eqref{e:maximal-conformal-embedding},  given by  a family of harmonic diffeomorphisms 
$$
F_\pm^\epsilon: \mathbb{D}^\epsilon\,  \longrightarrow \, \mathbb{D}^\epsilon_{\pm},
$$
where $\mathbb{D}^\epsilon:=\mathbb{D}_{z^\epsilon}$ and $\mathbb{D}^\epsilon_\pm:=\mathbb{D}_{z_\pm^\epsilon}$. Combining both cases, we introduce the following commutative diagram:
\begin{equation}\label{e:CD-2}
\begin{tikzcd}
\Large
\mathbb{D}_+ \arrow[d,swap, "h_+^\epsilon"] & \mathbb{D}  \arrow[l, swap, "F_+"]  \arrow[r, "F_-"]  \arrow[d, swap, "f^\epsilon"] &  \mathbb{D}_- \arrow[d, "h_-^\epsilon"] \\
\mathbb{D}_+^\epsilon & \mathbb{D}^\epsilon \arrow[l, "F^\epsilon_+"] \arrow[r, swap, "F^{\epsilon}_-"]   &  \mathbb{D}_-^\epsilon
\end{tikzcd}
\end{equation}
such that $h^\epsilon_{\pm}\circ F_\pm=F^\epsilon_{\pm}\circ f^\epsilon$.  Here  we use the notations $\mathbb{D}=\mathbb{D}^0$ and $\mathbb{D}_\pm=\mathbb{D}^0_{\pm}$ for simplicity.  Moreover, we assume that  $f^\epsilon$ and $h^\epsilon_{\pm}$ satisfy the corresponding equalities to \eqref{e:Bel-eq} and 
\eqref{e:der-Bel}. 

\begin{proposition}\label{p:var-Hopf-pm}
For a family of Hopf differentials $\Phi^\epsilon=\Phi^\epsilon_+$ of $F^\epsilon_+$ satisfying the commutative diagram \eqref{e:CD-2}, we have
\begin{equation}\label{e:var-Hopf-pm}
L_{\nu_f} \Phi = \frac{1}{2} e^\phi\Big( (F^*_+(\nu_+) - F^*_-(\nu_-)) \bar{\mu}_F^2+ (F^*_+(\bar{\nu}_+) - F^*_-(\bar{\nu}_-)) \Big),
\end{equation}
\begin{equation}\label{e:var-Bel-pm}
\nu_f= \frac12 (1+|\mu_F|^2)^{-1} \Big( (F^*_+(\nu_+) + F^*_-(\nu_-)) + (F^*_+(\bar{\nu}_+) + F^*_-(\bar{\nu}_-))\mu_F^2 \Big),
\end{equation}
where $\nu_{\pm}$ denotes the harmonic Beltrami differential $\nu_{h_\pm}$ and $\mu_F=\mu_{F_+}$.
\end{proposition}

\begin{proof}
By $\Phi_\pm= e^\phi \bar{\mu}_{F_\pm}$ and applying Proposition \ref{p:var-Hopf}  to $F^\epsilon_{\pm}$, we obtain
\begin{equation}\label{e:var-Hopf-app}
L_{{\nu}_f}\, \Phi_\pm = e^\phi \big(F_\pm^*({{\nu_\pm}})\bar{\mu}_{F_\pm}^2-{\bar{\nu}_f}|\mu_{F_\pm}|^2\big) 
+e^\phi \big( F^*_\pm(\bar{{\nu}}_\pm)-\bar{{\nu}}_f \big).
\end{equation}
Taking the difference of these equalities for $\Phi_\pm$ with $\Phi_+=-\Phi_-$, we derive equation \eqref{e:var-Hopf-pm}.
To prove \eqref{e:var-Bel-pm}, we apply Proposition \ref{p:var-nu-F} to $F^\epsilon_\pm$ and take their sum. This gives
\begin{equation}\label{e:var-nu-f-1}
2\nu_f-2\bar{\nu}_f\mu_F^2 =(1-|\mu_F|^2)\big( F^*_+(\nu_+)+F^*_-(\nu_-)\big).
\end{equation}
By combining this equation with its conjugated equality  and noting that  $\overline{F^*_\pm(\nu_\pm)}= F^*_\pm(\bar{\nu}_\pm)$, we obtain equation \eqref{e:var-Bel-pm}.
\end{proof}

\begin{remark}
In the identities \eqref{e:var-Hopf-pm} and \eqref{e:var-Bel-pm}, the pullbacks of the Beltrami differentials $\nu_\pm$, $\bar{\nu}_\pm$ by $F_\pm$, as well as
the Beltrami differentials $\mu_{F_\pm}$, all share the same tensor type $(-1,1)$.  However,  their  roles and dependencies  differ fundamentally. The Beltrami differential $\mu_{F_\pm}$ 
depends solely on the map $F_\pm:\mathbb{D}_z \to \mathbb{D}_{z_\pm}$, and plays the role of a coordinate over $T_0(1)$.
In contrast, the pullbacks of $\nu_\pm$, $\bar{\nu}_\pm$ by $F_\pm$  arise from the deformation and
represent tangent vectors on $T_0(1)$, thereby playing the role of vector fields.
\end{remark}

\subsection{Symplectic forms}

As described in Subsection \ref{ss:Hilbert},  for a point $\mu\in T_0(1)$, the local coordinate system over an open neighborhood $V_\mu$ around  $\mu$ is modeled on the Hilbert space $A_2(\mathbb{D}^*)$. More precisely,  fixing an orthonormal basis $\{\phi_k\}_{k\in\mathbb{Z}}$ of $A_2(\mathbb{D}^*)$,  the inverse images of $\phi_k$'s  under the map given in \eqref{e:coordinate-map}
determine the coordinate system $\{\mu_k\}_{k\in\mathbb{Z}}$ over $V_\mu$.  Thus, any point $\nu\in V_\mu$
can be expressed as
$$
\nu= \sum_{k\in\mathbb{Z}} \zeta_k\, \mu_k,
$$
for small complex  coefficients $\{\zeta_k\}_{k\in\mathbb{Z}}$. 

Similarly, for a point $(\mu, \Phi)$ in the holomorphic cotangent bundle $T^*T_0(1)$, the local coordinate system on an
open neighborhood around $(\mu,\Phi)$  is given by
$$\{\zeta_k, \eta_k\}_{k\in\mathbb{Z}}.$$
Here the fiber coordinates $\{\eta_k\}_{k\in\mathbb{Z}}$ are also determined by  the fixed orthonormal basis
of $A_2(\mathbb{D}^*)$, via the identification $T^*_{[\mu]} T_0(1) \cong A_2(\mathbb{D}^*)$.  Consequently,  the differentials $(d\zeta_k,d\eta_k)_{k\in\mathbb{Z}}$ provide a basis for the cotangent space $T^*_{(\mu,\Phi)}(T^*T_0(1))$. Note that $d\zeta_k$ can be represented by  harmonic Beltrami differentials,  while $d\eta_k$ corresponds to holomorphic quadratic differentials.
%Based on this consideration, the complex canonical symplectic form $\omega_{\mathbb{C}}$ on $T^*T_0(1)$ is given by
%\begin{equation}\label{e:complex-canonical}
% \omega_{\mathbb{C}}=\int_{\mathbb{D}}\, \sum_{k\in\mathbb{Z}} d\zeta_k\wedge d\eta_k\, d^2z.    
%\end{equation}

We now examine how the quantities $d\zeta_k$ and $d\eta_k$ arise in the context of the deformations of the maximal
conformal embeddings described in \eqref{e:maximal-conformal-embedding}. To this end, let us begin by considering the following set up.

Let $F_\pm:\mathbb{D}_z \to \mathbb{D}_\pm$  be harmonic diffeomorphisms as given in the commutative diagram \eqref{e:CD-2},
and denote by $\mu_\pm$ the Beltrami differentials representing $\mathbb{D}_{\pm }$ respectively. As described in Subsection  \ref{ss:Hilbert},
we fix an orthonormal basis  in $H^{-1,1}(\mathbb{D})$ and obtain the corresponding orthonormal basis for $T_{[\mu]} T_0(1)$ by applying right translations: $D_0 R_{[\mu]} (H^{-1,1}(\mathbb{D}))$ for any $[\mu]\in T_0(1)$.
We denote by $\{\nu_{\pm,k}\}_{k\in\mathbb{Z}}$ the resulting orthonormal basis of the tangent space $T_{[\mu_\pm]}T_0(1)$. 

Given a harmonic Beltrami differential
$\nu_{\pm,k}$ on $\mathbb{D}_{\pm}$, we solve the Beltrami  equation 
\begin{equation}
(h^\epsilon_{\pm})_{\bar{z}} =\epsilon \nu_{\pm,k} (h^\epsilon_{\pm})_z \qquad \text{where} \quad z=z_\pm,
\end{equation}
to obtain a quasi-conformal map $h^\epsilon_{\pm}:\mathbb{D}_\pm\to\mathbb{D}_\pm^\epsilon$. Note that $h^\epsilon_\pm$  and
$\mathbb{D}^\epsilon_\pm$ depend on index $k$, though
we supress this dependence in the notation for simplicity.

Given the two deformed discs $\mathbb{D}^\epsilon_\pm$ depending on the same index $k$, 
we construct a maximal conformal embedding  $\sigma^{\epsilon} :\mathbb{D}^{\epsilon}\to \mathbb{A}\mathrm{d}\mathbb{S}^{2,1}$, as the inverse of the $\mathrm{Mess}$ map,  such that the resulting induced Gauss maps coincide with the given harmonic maps
$$
F_\pm^\epsilon:\mathbb{D}^\epsilon\to \mathbb{D}^\epsilon_\pm.
$$ Correspondingly, there exists a unique quasi-conformal map 
$$
f^\epsilon:\mathbb{D}\to \mathbb{D}^{\epsilon}
$$
satisfying the relations
\begin{equation}
h^\epsilon_+\circ F_+ = F^\epsilon_+\circ f^\epsilon, \qquad  h^\epsilon_- \circ F_- = F^\epsilon_-\circ f^\epsilon.
\end{equation}
This follows from the fact that the deformed disk $\mathbb{D}^{\epsilon}$ is realized as the graph of the map $F^{\epsilon}_+\circ (F^{\epsilon}_-)^{-1}$ in $\mathbb{D}_+^{\epsilon}\times \mathbb{D}_-^{\epsilon}$,
and each point in $\mathbb{D}^\epsilon$ is determined by its coordinates via the maps $F^\epsilon_\pm$.
Note that the both the map $f^\epsilon$ and the deformed disc $\mathbb{D}^\epsilon$ depend on the index $k$,
although we omit this dependence in the notation for simplicity. 

We now define $\nu^\epsilon_k$ as the Beltrami differential on $\mathbb{D}$ corresponding to $f^\epsilon_k$, that is,
$$
(f^\epsilon_k)_{\bar{z}} =  \nu^{\epsilon}_k\, (f^\epsilon_k)_z.
$$
Let $\nu_k$ denotes the harmonic Beltrami differential given by the derivative of $\nu_k^\epsilon$ at $\epsilon=0$.
By construction, the basis $\{\nu_k\}_{k\in\mathbb{Z}}$  coincides with the one determined by $\{\nu_{\pm, k}\}$ via the right translations.
We define the variation of $\mu$ along $\nu_k$ by
\begin{equation}\label{e:def-var-mu}
\delta_k \mu:={\nu_k}=\frac12 (1+|\mu_F|^2)^{-1} \Big( (F^*_+(\nu_{+,k}) + F^*_-(\nu_{-,k})) 
+ (F^*_+(\bar{\nu}_{+,k}) + F^*_{-}(\bar{\nu}_{-,k}) )\mu_F^2 \Big)
\end{equation}
where the second equality follows from  equation\eqref{e:var-Bel-pm}. 
By \eqref{e:var-nu-F}, we  observe that $F^*_\pm(\nu_{\pm,k})$ is determined by $\nu_k$ and $\mu_F=\mu_{F_+}$.
Similarly,  the variation of $\Phi$ along $\nu_k$ is defined by
\begin{equation}\label{e:def-var-phi}
\delta_k\Phi:=L_{\nu_k}\Phi= \frac{1}{2} e^\phi\Big( (F^*_+(\nu_{+,k}) - F^*_-(\nu_{-,k})) \bar{\mu}_F^2+ (F^*_+(\bar{\nu}_{+,k}) - F^*_-(\bar{\nu}_{-,k})) \Big)
\end{equation}
where the second equality follows from  equation\eqref{e:var-Hopf-pm}.

\begin{proposition}\label{p:diffeo}
The differential of  the $\mathrm{Mess}$ map
$$\mathrm{Mess}:T^*T_0(1)\to T_0(1)\times T_0(1)$$
is an isomorphism.
\end{proposition}

\begin{proof}
By the proof of Theorem \ref{t:Mess-diff}, the $\mathrm{Mess}$ map is bijective, implying that its inverse $\mathrm{Mess}^{-1}$ is well-defined.
We now show that the differential of $\mathrm{Mess}^{-1}$ is an isomorphism. First, we observe that the differential of $\mathrm{Mess}^{-1}$
is given by \eqref{e:def-var-mu} and \eqref{e:def-var-phi}.
To analyze these maps, we decompose the tangent space $T(T_0(1)\times T_0(1))$ at
$z_\pm|_{\mathbb{S}^1}$ into the tangent space of the diagonal and the tangent space of  the anti-diagonal.
Along the tangent space of the diagonal, we have $F^*_+(\nu_{+,k})= F^*_-(\nu_{-,k})$.  Thus, the component of the differential of $\mathrm{Mess}^{-1}$ given in \eqref{e:def-var-phi} vanishes over this subspace. If the component of the differential of $\mathrm{Mess}^{-1}$
given in \eqref{e:def-var-mu} has a nontrivial kernel, then there must exist a harmonic Beltrami differential $\nu_+$ such that
\begin{equation*}
F^*_+(\nu_+)+F^*_+(\bar{\nu}_+)\mu_F^2=0.
\end{equation*}
Combining this with its complex conjugate equation, we obtain $F^*_+(\nu_+) (1-|\mu_F|^4)=0$, which can not occur
since $|\mu_F|<1$ and $F_+$ is a diffeomorphism. Hence, the component of the differential of $\mathrm{Mess}^{-1}$
given in \eqref{e:def-var-mu} is injective. 

To show that this component is surjective over the tangent space of the diagonal, let $\nu$ be a harmonic Beltrami differential
representing a tangent vector $\delta\mu$,  Then, there exists a Beltrami differential $\nu_+$ satisfying
\begin{equation*}
\nu-\bar{\nu}\mu_F^2 = (1-|\mu_F|^2) F_+^*(\nu_+),
\end{equation*}
by the fact $F$ is a diffeomorphism. This implies 
\begin{equation*}
\nu= (1+|\mu_F|^2)^{-1} \Big( F^*_+(\nu_+)+ F^*_+(\bar{\nu}_+)\mu_F^2 \Big).
\end{equation*}
Since $\nu$ can be written in terms of the basis $\nu_k$'s, the same holds for $\nu_+$ in terms of the basis  $\nu_{+,k}$'s.  Thus, this component is surjective over
the tangent space of the diagonal.

Similarly, we can show that the differential of $\mathrm{Mess}^{-1}$ given in \eqref{e:def-var-phi} is bijective
over the tangent space of the anti-diagonal.  Therefore, combining these results, we conclude that the differential of $\mathrm{Mess}:T^*T_0(1)\to T_0(1)\times T_0(1)$ is an isomorphism.
This completes the proof.
\end{proof}

By Proposition \ref{p:diffeo}, we conclude that  the families $\{\delta_k\mu \}_{k\in\mathbb{Z}}$ and $\{\delta_k\Phi\}_{k\in\mathbb{Z}}$ together
form a basis for the cotangent bundle of $T^*T_0(1)$ at $(\mu,\Phi)$.  Therefore, they represent the realization of the basis
$\{d\zeta_k, d\eta_k\}_{k\in\mathbb{Z}}$ in the context of a deformation of the maximal conformal embedding \eqref{e:maximal-conformal-embedding}. 

\begin{remark}
Strictly speaking, the basis $\{d\zeta\}_{k\in\mathbb{Z}}$  is realized by the duals of the variations $\{\delta_k\mu=\nu_k\}_{k\in\mathbb{Z}}$ rather than the variations themselves. However,  for simplicity of notation, we will not distinguish between these in the subsequent constructions. 
\end{remark}

Based on this identification, the \emph{complex canonical symplectic form} $\omega_{\mathbb{C}}$ on $T^*T_0(1)$ admits the following expression:
\begin{equation*}
\omega_{\mathbb{C}}=\int_{\mathbb{D}}\ \sum_{k\in\mathbb{Z}} \delta_k\mu\wedge \delta_k\Phi\ d^2z.
\end{equation*}
Moreover, its imaginary part, referred to as the \emph{canonical symplectic form},  is given by
\begin{equation}\label{e:symplectic-def}
\omega_{\mathrm{C}}=\int_{\mathbb{D}} \Big( \sum_{k\in\mathbb{Z}} \sqrt{-1}\delta_k\Phi\wedge \delta_k \mu- \sqrt{-1}\delta_k\overline{\Phi}\wedge \delta_k \bar{\mu}\Big)\, d^2z.
\end{equation}
Note that the wedge product $\wedge$ is applied to the vector-valued expressions $F^*_{\pm}(\nu_{\pm,k})$ appearing on the right hand sides
of \eqref{e:def-var-mu} and \eqref{e:def-var-phi}.
The construction of $\omega_{\mathrm{C}}$ is globally well-defined, since all the local expressions are compatible through  right translations from a neighborhood $V_0$ of the origin in $T_0(1)$. 
While the definition of $\omega_{\mathrm{C}}$ may seem to depend on
the choice of basis $\{\nu_{\pm,k}|\, k\in\mathbb{Z}\}$, we will later show that $\omega_{\mathrm{C}}$ is in fact  independent of this choice
(see Remark \ref{r:independence-basis}).

\begin{proposition}\label{p:symplectic-form}
The following equality holds:
\begin{equation}\label{e:symplectic-form}
\begin{split}
\omega_{\mathrm{C}} =
&-\frac{\sqrt{-1}}{2}\int_{\mathbb{D}} (1-|\mu_{F_+}|^2)\, e^{\phi} \\
&\qquad\qquad \cdot \, \sum_{k\in\mathbb{Z}} \Big( F^*_+(\nu_{+,k})\wedge F^*_+(\bar{\nu}_{+,k}) -
F^*_-(\nu_{-,k})\wedge F^*_-(\bar{\nu}_{-,k}) \Big)\, d^2z.
\end{split}
\end{equation}
\end{proposition}

\begin{proof}
By \eqref{e:symplectic-def}, \eqref{e:var-Hopf-pm}, and \eqref{e:var-Bel-pm}, and recalling that $\mu_F=\mu_{F_+}$, we have
\begin{align*}
&\sqrt{-1}\Big( \sum_{k\in\mathbb{Z}} \delta_k\Phi\wedge \delta_k\mu -\delta_k\overline{\Phi}\wedge \delta_k\bar{\mu}\Big) \\
=& \frac{\sqrt{-1}}4 (1+|\mu_F|^2)^{-1} \, e^{\phi}\, \sum_{k\in\mathbb{Z}}\Big( (F^*_+(\nu_{+,k}) - F^*_-(\nu_{-,k})) \bar{\mu}_F^2+ (F^*_+(\bar{\nu}_{+,k}) - F^*_-(\bar{\nu}_{-,k}) \Big)\\
&\qquad\qquad\qquad\qquad\qquad \qquad  \wedge \Big( (F^*_+(\nu_{+,k}) + F^*_-(\nu_{-,k})) + (F^*_+(\bar{\nu}_{+,k}) + F^*_-(\bar{\nu}_{-,k})\mu_F^2 \Big)\\
& -\frac{\sqrt{-1}}4 (1+|\mu_F|^2)^{-1} \, e^{\phi}\, \sum_{k\in\mathbb{Z}}\Big( (F^*_+(\bar{\nu}_{+,k}) - F^*_-(\bar{\nu}_{-,k})) {\mu}_F^2+ (F^*_+({\nu}_{+,k}) - F^*_-({\nu}_{-,k}) \Big)\\
&\qquad\qquad\qquad\qquad\qquad  \qquad\quad  \wedge \Big( (F^*_+(\bar{\nu}_{+,k}) + F^*_-(\bar{\nu}_{-,k})) + (F^*_+({\nu}_{+,k}) + F^*_-({\nu}_{-,k})\bar{\mu}_F^2 \Big).
\end{align*}
Rearranging the terms, we obtain
\begin{align*}&\sqrt{-1}\Big( \sum_{k\in\mathbb{Z}} \delta_k\Phi\wedge \delta_k\mu -\delta_k\overline{\Phi}\wedge \delta_k\bar{\mu}\Big) \\
=& \frac{\sqrt{-1}}4 (1+|\mu_F|^2)^{-1}\, e^{\phi}\, \Big[ \sum_{k\in\mathbb{Z}}
(F^*_+(\bar{\nu}_{+,k}) - F^*_-(\bar{\nu}_{-,k}) \wedge  (F^*_+({\nu}_{+,k}) + F^*_-({\nu}_{-,k}) )\\
&\qquad\qquad\qquad\qquad  \qquad + (F^*_+(\nu_{+,k}) - F^*_-(\nu_{-,k}))  \wedge  (F^*_+(\bar{\nu}_{+,k}) + F^*_-(\bar{\nu}_{-,k})|\mu_F|^4\\
&\qquad\qquad\qquad\qquad  \qquad - (F^*_+({\nu}_{+,k}) - F^*_-({\nu}_{-,k}) )\wedge (F^*_+(\bar{\nu}_{+,k}) + F^*_-(\bar{\nu}_{-,k})) \\
&\qquad\qquad\qquad\qquad  \qquad - (F^*_+(\bar{\nu}_{+,k}) - F^*_-(\bar{\nu}_{-,k})) \wedge  (F^*_+({\nu}_{+,k}) + F^*_-({\nu}_{-,k})|{\mu}_F|^4 
\Big]\\
=& \frac{\sqrt{-1}}4 (1-|\mu_F|^2)\, e^{\phi}\, \Big[ \sum_{k\in\mathbb{Z}}(F^*_+(\bar{\nu}_{+,k}) - F^*_-(\bar{\nu}_{-,k}) \wedge  (F^*_+({\nu}_{+,k}) + F^*_-({\nu}_{-,k}) )\\
&\qquad\qquad\qquad\qquad  \qquad - (F^*_+({\nu}_{+,k}) - F^*_-({\nu}_{-,k}) )\wedge (F^*_+(\bar{\nu}_{+,k}) + F^*_-(\bar{\nu}_{-,k}))\Big]\\
=&-\frac{\sqrt{-1}}2 (1-|\mu_F|^2)\, e^{\phi}\, \Big[ \sum_{k\in\mathbb{Z}} F^*_+({\nu}_{+,k})\wedge F^*_+(\bar{\nu}_{+,k}) -  F^*_-({\nu}_{-,k}) \wedge  F^*_-(\bar{\nu}_{-,k})  \Big].
\end{align*}
This completes the proof.
\end{proof}

Note that for the orthonormal  basis $\{\nu_{\pm,k}\}_{k\in\mathbb{Z}}$  with respect to the Weil-Petersson inner product on  $T_{[\mu_{\pm}]} T_0(1)$,  
the Weil-Petersson symplectic form $\omega_{\mathrm{WP}}$ on $T_0(1)$ is given by
\begin{equation}
\omega_{\mathrm{WP}}= \frac{\sqrt{-1}}{2} \int_{\mathbb{D}} e^{\psi}\big( \sum_{k\in\mathbb{Z}} \nu_{\pm,k} \wedge \bar{\nu}_{\pm,k}\big) \, d^2z.
\end{equation}

\begin{theorem}\label{t:symplectic-relation}
The map 
$$\mathrm{Mess}:T^*T_0(1)\to T_0(1)\times T_0(1)$$
is a symplectic diffeomorphism such that
\begin{equation}\label{e:symplectic-relation}
\omega_{\mathrm{C}} =  -\mathrm{Mess}^*_+ (\omega_{\mathrm{WP}}) + \mathrm{Mess}^*_-(\omega_{\mathrm{WP}}).
\end{equation}
Here, $\mathrm{Mess}_\pm:= \pi_{\pm}\circ \mathrm{Mess}$, where $\pi_\pm$ denotes the projection map from $T_0(1)\times T_0(1)$
onto the first or second factor, respectively.
\end{theorem}
\begin{proof}
By Proposition \ref{p:symplectic-form}, we have
\begin{equation}\label{e:w-cot}
\begin{split}
\omega_{\mathrm{C}}
=&-\frac{\sqrt{-1}}2 \int_{\mathbb{D}}(1-|\mu_F|^2) \, e^{\phi}\\
&\qquad\qquad\quad\cdot \Big[ \sum_{k\in\mathbb{Z}} F^*_+({\nu}_{+,k})\wedge F^*_+(\bar{\nu}_{+,k})
  -  F^*_-({\nu}_{-,k}) \wedge  F^*_-(\bar{\nu}_{-,k})  \Big] \, d^2 z\\
=&-\frac{\sqrt{-1}}2 \int_{\mathbb{D}}\, e^{\psi\circ F_\pm} \big(|(F_\pm)_z|^2 -|(F_\pm)_{\bar{z}}|^2\big)\\ 
&\qquad\qquad\quad  \cdot \Big[ \sum_{k\in\mathbb{Z}} F^*_+({\nu}_{+,k})\wedge F^*_+(\bar{\nu}_{+,k}) 
  -  F^*_-({\nu}_{-,k}) \wedge  F^*_-(\bar{\nu}_{-,k})  \Big]
\, d^2z\\
=&-\frac{\sqrt{-1}}2\Big[ \mathrm{Mess}^*_+\int_{\mathbb{D}} e^{\psi} \, \Big(\sum_{k\in\mathbb{Z}} \nu_{+,k} \wedge \bar{\nu}_{+,k} \Big)\, d^2z\\
  &\qquad\qquad \qquad\qquad \qquad -\mathrm{Mess}^*_-\int_{\mathbb{D}} e^{\psi} \Big(\sum_{k\in\mathbb{Z}} \nu_{-,k} \wedge \bar{\nu}_{-,k} \Big) \, d^2z \Big].
\end{split}
\end{equation}
Thus, we obtain
\begin{align*}
\omega_{\mathrm{C}}
=  -\mathrm{Mess}^*_+ (\omega_{\mathrm{WP}}) + \mathrm{Mess}^*_-(\omega_{\mathrm{WP}}).
\end{align*}
This completes the proof.
\end{proof}

\begin{remark}\label{r:independence-basis}
By Theorem \ref{t:symplectic-relation}, we can see that the canonical symplectic form $\omega_{\mathrm{C}}$ does not depend on the choice of the basis $\{\nu_{\pm,k} :\, k\in\mathbb{Z}\}$.
\end{remark}

\begin{remark}\label{r:bounded}
For any harmonic Beltrami differentials $\mu_1, \mu_2$ representing vectors in  the tangent space of $T_\nu T_0(1)\cong H^{-1,1}(\mathbb{D})$ at  $[\nu]\in T_0(1)$, we easily obtain the inequality 
$$
\big|\,\omega_{\mathrm{WP}}(\mu_1,\mu_2)\, \big|^2  \leq ||\mu_1||_{2}\cdot||\mu_2||_2 < \infty
$$
where $||\cdot||_2$ denotes the $L^2$-norm on $H^{-1,1}(\mathbb{D})$.  This observation, together with Theorem \ref{t:symplectic-relation},   implies that the canonical symplectic form $\omega_{\mathrm{C}}$ on $T^*T_0(1)$ satisfies the same boundedness  property.
\end{remark}

\begin{remark}\label{r:pullback-form}
From the proof of Theorem \ref{t:symplectic-relation}, we observe that the pullback $\mathrm{Mess}_\pm^* (\omega_{\mathrm{WP}})$ has the following expression:
\begin{equation}\label{e:symplectic-pm}
\mathrm{Mess}^*_\pm (\omega_{\mathrm{WP}})= \frac{\sqrt{-1}}{2} \int_{\mathbb{D}} (1-|\mu_F|^2)\, e^\phi \, \Big(\sum_{k\in\mathbb{Z}} F^*_\pm(\nu_{\pm,k})\wedge F^*_\pm(\bar{\nu}_{\pm,k}) \Big) \, d^2z.
\end{equation}
The closedness of these pullback 2-forms on $T^*T_0(1)$ 
follows from the commutativity of  the exterior differential $d$ with the pullback operation.  Alternatively, it can be established 
using  similar arguments as in Lemma 2.7 of \cite{Wol86} or Theorem 7.4 of \cite{TT06}, which in turn originate from \cite{Ahl61}.  A key component of  this argument is the fact that
the Lie derivative of the density $(1-|\mu_F|^2) e^\phi$ vanishes,  a result that follows from
Propositions \ref{p:var-anti-hol} and \ref{p:var-phi}.
\end{remark}

\begin{corollary}
For the symplectic diffeomorphism  $\mathrm{Mess}:T^*T_0(1)\to T_0(1)\times T_0(1)$, the following properties hold:
\begin{enumerate}
\item{The $\mathrm{Mess}$ map  sends the Lagrangian subspace given by the zero section of $T^*T_0(1)$ to the diagonal subset in $T_0(1)\times T_0(1)$,}
\item{The $\mathrm{Mess}$ map sends the Lagrangian subspace given by the fiber $T^*_0T_0(1)$ to the anti-diagonal subset in $T_0(1)\times T_0(1)$.}
\end{enumerate}
\end{corollary}

\begin{proof}
The first claim follows directly from Remark \ref{r:totally-geodesic}.
For the second claim concerning the Lagrangian subspace of the fiber $T_0^*T_0(1)$,  we observe that the disc $\mathbb{D}_z$ coincides with $\mathbb{D}_w$,  meaning that  $z_\pm=F_\pm$.
Consequently, the Beltrami differentials associated with $z_\pm$ are identical to the Beltrami differential $\mu_{F_\pm}$.  Furthermore,  we have $\mu_{F_\pm} = \pm \overline{\Phi} e^{-\phi}$.
Thus, $\mathrm{Mess}$ maps the Lagrangian subspace corresponding to the fiber of $T_0^*T_0(1)$ to the anti-diagonal subset in $T_0(1)\times T_0(1)$.
\end{proof}

\section{Variation of anti-holomorphic energy}\label{s:anti-holomorphic}

In this section, we prove that the anti-holomorphic energy functional $E$
serves as a  K\"ahler potential function for  the canonical symplectic form $\omega_C$ when restricted to the submanifolds $T_0(1)^\pm$
of $T^*T_0(1)$.  
Although $T_0(1)^\pm$  are real symplectic submanifolds of the complex manifold $T^*T_0(1)$, 
they are endowed with a complex structure by identifying them with $T_0(1)$.
This identification allows us to interpret $E$ as a K\"ahler potential with respect to the
induced K\"ahler structure on $T_0(1)^\pm$.

Throughout the remainder of this section, we slightly abuse notation by writing
$F=F_\pm$ whenever no confusion arises.
We begin by considering variations of the holomorphic and anti-holomorphic energy densities.

\begin{proposition}\label{p:var-anti-hol}
For a family of anti-holomorphic energy densities $e^{\psi^\epsilon\circ F^\epsilon}|F^\epsilon_{\bar{z}}|^2$ on $\mathbb{D}^\epsilon$ satisfying the commutative diagram \eqref{e:CD-1},
\begin{equation}\label{e:var-phi-0}
L_{{\nu}_f}\big(\,e^{\psi\circ F}|F_{\bar{z}}|^2\, \big) = \Phi \big(F^*({{\nu_h}})-{\nu_f}\big) +\overline{\Phi} \big( F^*(\bar{{\nu}}_h)-\bar{{\nu}}_f\big).
\end{equation}
\end{proposition}

\begin{proof}
From the definition of Lie derivative, we have
\begin{equation}\label{e:var-phi-1}
L_{{\nu}}\,\big(\,e^{\psi\circ F}|F_{\bar{z}}|^2\, \big)  =\frac{\partial}{\partial \epsilon}\Big|_{\epsilon=0} \Big( e^{\psi^\epsilon\circ F^\epsilon\circ f^\epsilon} F^\epsilon_{\bar{z}}\circ f^\epsilon\, \overline{F}_{{z}}^{\epsilon}\circ f^{\epsilon}\, f^{\epsilon}_z \bar{f}^\epsilon_{\bar{z}} \Big).
\end{equation}
From $h^\epsilon \circ F= F^{\epsilon}\circ f^{\epsilon}$, we obtain
\begin{equation}\label{e:var-phi-2}
\frac{\partial}{\partial\epsilon}\Big|_{\epsilon=0} \Big(F^\epsilon_{\bar{z}}\circ f^{\epsilon} \bar{f}^\epsilon_{\bar{z}}\Big) =(\dot{h}\circ F)_{\bar{z}}- F_{{z}} \dot{{f}}_{\bar{z}}, \qquad \frac{\partial}{\partial\epsilon}\Big|_{\epsilon=0} \Big(\overline{F}^\epsilon_{{z}}\circ f^{\epsilon} {f}^\epsilon_{{z}}\Big) =(\dot{\bar{h}}\circ F)_{{z}}- \overline{F}_{\bar{z}} \dot{\bar{f}}_{{z}}
\end{equation}
By combining  these expressions with the previous equation, we get
\begin{align*}
L_{{\nu_f}}\, \big(\,e^{\psi\circ F}|F_{\bar{z}}|^2\, \big) =&\ e^{\psi\circ F} \Big( \dot{\psi}+\psi_u\dot{h}+\psi_{\bar{u}}\dot{\bar{h}}\Big)\circ F F_{\bar{z}} \overline{F}_{{z}} \\
&+ e^{\psi\circ F} \Big( (\dot{h}\circ F)_{\bar{z}} - F_{{z}} \dot{{f}}_{\bar{z}}\Big)\overline{F}_{{z}}
+  e^{\psi\circ F} F_{\bar{z}} \Big( (\dot{\bar{h}}\circ F)_{{z}} - \overline{F}_{\bar{z}} \dot{\bar{f}}_{{z}}\Big)\\
=&\ e^{\psi\circ F} \Big( \dot{\psi}+\psi_u\dot{h}+\psi_{\bar{u}}\dot{\bar{h}}+ \dot{h}_z + \dot{\bar{h}}_{\bar{z}} \Big)\circ F F_{\bar{z}} \overline{F}_{{z}}\\
&+e^{\psi\circ F} \Big( \dot{h}_{\bar{z}}\circ F \overline{F}_{\bar{z}} \overline{F}_{{z}} - \dot{{f}}_{\bar{z}}  F_{{z}} \overline{F}_{{z}} 
     + \dot{\bar{h}}_{{z}}\circ F {F}_{{z}} {F}_{\bar{z}} - \dot{\bar{f}}_{{z}}  F_{\bar{z}} \overline{F}_{\bar{z}} \Big)
\end{align*}
Using the equality \eqref{e:Ahl} for a variation of  hyperbolic metrics, we conclude
\begin{align*}
L_{{\nu_f}}\, \big(\,e^{\psi\circ F}|F_{\bar{z}}|^2\, \big)
=&\ e^{\psi\circ F} F_z\overline{F}_z \Big( \dot{h}_{\bar{z}}\circ F \overline{F}_{\bar{z}} F_z^{-1} -\dot{f}_{\bar{z}}\Big) +
e^{\psi\circ F} F_{\bar{z}}\overline{F}_{\bar{z}} \Big( \dot{\bar{h}}_z\circ F F_z \overline{F}_{\bar{z}}^{-1} -\dot{\bar{f}}_z\Big)\\
=&\ \Phi \Big( F^*({\nu}_h) -{\nu}_f \Big)+ \overline{\Phi} \Big( F^*({\bar{\nu}}_h) - {\bar{\nu}_f} \Big).
\end{align*}
This completes the proof.
\end{proof}

In the same way as above, we can prove the following proposition.

\begin{proposition}\label{p:var-phi}
For a family of holomorphic energy densities $e^{\phi^\epsilon}= e^{\psi^\epsilon\circ F^\epsilon}|F^\epsilon_z|^2$ on $\mathbb{D}^{\epsilon}$ satisfying the commutative diagram \eqref{e:CD-1}, the following equality holds:
\begin{equation}\label{e:var-phi}
L_{{\nu}_f}\, e^{\phi} = \Phi \big(F^*({{\nu_h}})-{\nu_f}\big) +\overline{\Phi} \big( F^*(\bar{{\nu}}_h)-\bar{{\nu}}_f\big).
\end{equation}
\end{proposition}

 For the following two propositions, we assume that either $F^*_+(\nu_+)$ or $F^*_-(\nu_-)$ vanishes, holding 
respectively over $T_0(1)^-$ and $T_0(1)^+$.

\begin{proposition}\label{p:hol-var}  Under the condition $F_+^*(\nu_+)=0$ or $F_-^*(\nu_-)=0$, we have
\begin{equation}
\frac12\big(L_\nu -i L_{i\nu}\big) \big( e^{\psi\circ F}|F_{\bar{z}}|^2\big)= e^\phi \bar{\mu}_F \nu= \Phi \, \nu
\end{equation}
\end{proposition}

\begin{proof}
We will prove the statement for the case $F_-^*(\nu_-)=0$ as the proof for the other case follows similarly.
From \eqref{e:var-phi-0}, we obtain
\begin{equation}\label{e:var-phi+}
L_{{\nu}} \big(e^{\psi\circ F}|F_{\bar{z}}|^2\big) = \Phi \big(F^*({{\nu_h}})-{\nu}\big)+\overline{\Phi} \big(F^*({\bar{\nu}_h})-\bar{\nu}\big) .
\end{equation}
By \eqref{e:var-nu-f-1} and the condition $F_-^*(\nu_-)=0$, we have
\begin{equation}\label{e:var-nu+}
F^*_+(\nu_+)=\, 2(1-|\mu_F|^2)^{-1}\big(\nu-\bar{\nu}\mu_F^2\big).
\end{equation}
Substituting this into the previous equation, we get
\begin{align*}
&L_{\nu} \big( e^{\psi\circ F}|F_{\bar{z}}|^2 \big)\\
=& (1-|\mu_F|^2)^{-1}\Big(\Phi \big(2\nu-2\bar{\nu}\mu_F^2 -\nu+\nu|\mu_F|^2\big)+\overline{\Phi}  \big(2\bar{\nu}-2{\nu}\bar{\mu}_F^2 -\bar{\nu}+\bar{\nu}|\mu_F|^2\big)\Big) \\
=& e^{\phi}   (1-|\mu_F|^2)^{-1} \Big( \nu\bar{\mu}_F+\nu\bar{\mu}_F |\mu_F|^2-2\bar{\nu}{\mu}_F|\mu_F|^2
 +\bar{\nu}{\mu}_F+\bar{\nu}{\mu}_F |\mu_F|^2-2{\nu}\bar{\mu}_F|\mu_F|^2\Big)\\
=& e^{\phi}   (1-|\mu_F|^2)^{-1} \Big(  \nu\bar{\mu}_F-\bar{\nu}{\mu}_F|\mu_F|^2
 +\bar{\nu}{\mu}_F-{\nu}\bar{\mu}_F|\mu_F|^2 \Big)\\
=& e^{\phi} \Big( \nu\bar{\mu}_F +\bar{\nu}{\mu}_F\Big).
\end{align*}
Thus, we obtain
\begin{align*}
\frac12\big(L_\nu -i L_{i\nu}\big)  \big( e^{\psi\circ F}|F_{\bar{z}}|^2 \big)
=\frac12 e^\phi  \Big( \nu\bar{\mu}_F +\bar{\nu}{\mu}_F +\nu\bar{\mu}_F -\bar{\nu}{\mu}_F\Big)
= e^\phi \nu\bar{\mu}_F.
\end{align*}
This completes the proof.

\end{proof}

\begin{proposition}\label{p:antihol-var}
Under the condition  $F_+^*(\nu_+)=0$ or $F_-^*(\nu_-)=0$, we have
\begin{equation}
\frac12\big(L_\mu+iL_{i\mu}\big) \big(e^\phi \bar{\mu}_F\, \nu\big) =e^{\phi}(1+|\mu_F|^2)\nu\bar{\mu}.
\end{equation}
\end{proposition}

\begin{proof}
We will prove the statement for the case $F_-^*(\nu_-)=0$ as the proof for the other case follows similarly.
From \eqref{e:var-Hopf} and $L_\mu\nu=0$ (see (2.3) of \cite{TZCMP}), we obtain
\begin{align*}
&L_\mu\big( e^\phi\bar{\mu}_F\, \nu\big)= L_\mu\big(\Phi \nu\big)\\ 
=&\, e^\phi  \big(F^*({{\mu_h}})\bar{\mu}_F^2-{\bar{\mu}}|\mu_F|^2\big) + e^\phi \big( F^*(\bar{{\mu}}_h)-\bar{{\mu}} \big)\,\nu\\
=&\, e^\phi (1-|\mu_F|^2)^{-1} \Big( 2\mu\bar{\mu}_F^2 -2\bar{\mu}|\mu_F|^4 -\bar{\mu}|\mu_F|^2+\bar{\mu}|\mu_F|^4 \Big)\nu\\
&+e^\phi (1-|\mu_F|^2)^{-1} \Big( 2\bar{\mu} -2{\mu}\bar{\mu}_F^2 -\bar{\mu}+\bar{\mu}|\mu_F|^2 \Big)\nu\\
=& e^\phi(1-|\mu_F|^2)^{-1} \bar{\mu}(1-|\mu_F|^4) \nu =e^\phi(1+|\mu_F|^2) \bar{\mu} \nu .
\end{align*}
Hence, we obtain
\begin{equation*}
\frac12\big(L_\mu+iL_{i\mu}\big) \big(e^\phi\bar{\mu}_F\, \nu \big)=e^\phi(1+|\mu_F|^2) \nu \bar{\mu}.
\end{equation*} 
This completes the proof.
\end{proof}

\begin{remark}\label{r:subspace-pm}   
A point in the space $T_0(1)\times\{0\}$ or $\{0\}\times T_0(1)$ corresponds to a maximal conformal embedding
$\sigma:\mathbb{D}_z \to \mathbb{A}\mathrm{d}\mathbb{S}^{2,1}$ such that one of the target discs of $F_\pm$ is fixed to be the origin, that is, $\mathbb{D}_w$. Equivalently,  this means that one of $\mathbb{D}_{z_\pm}$ coincides with $\mathbb{D}_w$ in the commutative diagram \eqref{e:CD1}.  
 Hence, using \eqref{e:composition},  we can derive the following identities along $T_0(1)^{\pm}$:
\begin{equation}
\mu_{z_\pm}= \frac{2\mu_z}{1+|\mu_z|^2}, \qquad \mu_{F_\pm}= - (z^{-1})^*(\mu_z)
\end{equation}
where $\mu_{z_\pm}$, $\mu_z$ denotes the Beltrami differential of the maps $z_\pm$ and $z$ from $\mathbb{D}_w$.    

\end{remark}

\begin{remark}\label{r:Kahler submanifold}
By construction, the submanifolds $T_0(1)^\pm\subset T^*T_0(1)$ consist of pairs $(\mu,\Phi)\in T^*T_0(1)$ where $\Phi$ is the Hopf differential of the harmonic map $F_\pm$ from the disc $\mathbb{D}_z$, determined $\mu$, to a fixed target disc $\mathbb{D}_{z_\mp}=\mathbb{D}_w$, as noted in Remark \ref{r:subspace-pm}.
This geometric setup implies that as $\mu$ varies in $T_0(1)$, the corresponding $\Phi$ depend only on $\mu$, thereby defining a section of the cotangent bundle $T^*T_0(1)$. The images of these sections are precisely the submanifolds $T_0(1)^\pm$.  

However, such a section is not a holomorphic section,  as the anti-holomorphic derivative of $\Phi$ does not vanish. In fact,  one can easily verify that the image of a holomorphic section of $T^*T_0(1)$ is necessarily a Lagrangian submanifold with respect to $\omega_{\mathrm{C}}$.
In contrast,  $T_0(1)^\pm$ are real symplectic submanifolds of $T^*T_0(1)$, equipped with the restriction of $\omega_{\mathrm{C}}$. 

Beyond the case discussed in \cite{TT06} and referenced in the Introduction,  a similar phenomenon occurs
in the finite dimensional Teichm\"uller space setting, as explored in
\cite{TZ88b} and \cite{TT03}.  In those works, the corresponding submanifolds of the cotangent bundle arises as the differences between the Fuchsian projective structure
and either the Schottky or Bers projective structures. 

\end{remark}

 Finally,  we show that the anti-holomorphic energy functional of the harmonic map $F_\pm$ 
is a K\"ahler potential function of the canonical symplectic form $\omega_{\mathrm{C}}$ over $T_0(1)^\pm \subset T^*T_0(1)$:

\begin{theorem}\label{t:K-potential} 
Over the subspace $T_0(1)^\pm$ of $T^*T_0(1)$,
\begin{equation}
\partial\bar{\partial}\, \big(2E\big)=\partial\bar{\partial} \Big( 2 \int_{\mathbb{D}} e^{\psi\circ F}|F_{\bar{z}}|^2 \, d^2z \Big) = \mp \sqrt{-1}\, i^*_\pm \omega_{\mathrm{C}},
\end{equation}
where $i_\pm :T_0(1)^\pm\to T^*T_0(1)$ denotes the embedding map of $T_0(1)^\pm$, respectively.
\end{theorem}

\begin{proof}
We prove the case for  $T_0(1)^+$ as the proof for the other case follows similarly.

To analyze the variation of the anti-holomorphic energy of $F_\pm$, we consider the variation of the anti-holomorphic density $e^\phi|\mu_F|^2 \, d^2z$  along a family of quasi-conformal map $f^\epsilon$. This is given by
\begin{align*}
&- \frac{\sqrt{-1}}{2}\, {e^{\phi^{\epsilon}}\circ f^{\epsilon}\, |\mu_{F^{\epsilon}}|^2\circ f^{\epsilon} \,d f^{\epsilon}\wedge d\bar{f^\epsilon}}\\
=&- \frac{\sqrt{-1}}{2}\, e^{\phi^{\epsilon}}\circ f^{\epsilon}\,|\mu_{F^{\epsilon}}|^2\circ f^{\epsilon} \,\big( |f^\epsilon_z|^2 - |f^\epsilon_{\bar{z}}|^2 \big)\, dz\wedge d\bar{z} \\
=& - \frac{\sqrt{-1}}{2}\,  e^{\phi^{\epsilon}}\circ f^{\epsilon}\,|\mu_{F^{\epsilon}}|^2\circ f^{\epsilon} \, |f^\epsilon_z|^2 \big( 1- \mathrm{O}(\epsilon^2) |\mu_f|^2 \big)\, dz\wedge d\bar{z}
\end{align*}
where we used the equality $f^\epsilon_{\bar{z}} =  \mu_f^{\epsilon} f^\epsilon_z$ for the last equality.  Thus, the factor $|f^\epsilon_z|^2$ together
with the variational term of $e^\phi |\mu_F|^2$ precisely corresponds to the Lie derivative described in Propositions \ref{p:hol-var} and \ref{p:antihol-var}. Consequently, we obtain
\begin{equation}\label{e:dd-bar}
\begin{split}
&\partial_{\nu}  \Big(2 \int_{\mathbb{D}} e^{\psi\circ F}|F_{\bar{z}}|^2 \, d^2z \Big)=\,2\int_{\mathbb{D}}  e^\phi\, \bar{\mu}_F\, \nu \, d^2z = 2\int_{\mathbb{D}}  \Phi\, \nu \, d^2z,\\
&\bar{\partial}_\mu \partial_{\nu}  \Big(2 \int_{\mathbb{D}} e^{\psi\circ F}|F_{\bar{z}}|^2 \, d^2z \Big)=\,  2 \int_{\mathbb{D}}  e^\phi\, (1+|\mu_F|^2)\, \nu \bar{\mu} \, d^2z.
\end{split}
\end{equation}

By the condition $F_-^*(\nu_-)=0$ and \eqref{e:symplectic-form}, 

\begin{equation}\label{e:symplectic-form+}
\begin{split}
\omega_{\mathrm{C}} =&
-\frac{\sqrt{-1}}{2}\int_{\mathbb{D}} (1-|\mu_F|^2) \, e^{\phi}\, \sum_{k\in\mathbb{Z}} \Big( F^*_+(\nu_{+,k})\wedge F^*_+(\bar{\nu}_{+,k}) \Big)\, d^2z \\
=&
-\frac{\sqrt{-1}}{2}\int_{\mathbb{D}} (1-|\mu_F|^2) \, e^{\psi\circ F_+} |(F_+)_z|^2 \, \sum_{k\in\mathbb{Z}} \Big( F^*_+(\nu_{+,k})\wedge F^*_+(\bar{\nu}_{+,k}) \Big)\, d^2z\\
=&
-\frac{\sqrt{-1}}{2}\mathrm{Mess}^*_+\int_{\mathbb{D}}  \, e^{\psi}\, \sum_{k\in\mathbb{Z}} \Big( \nu_{+,k} \wedge \bar{\nu}_{+,k} \Big)\, d^2z\\
=&
-\frac{\sqrt{-1}}{2}\mathrm{Mess}^*_+\int_{\mathbb{D}}  \, e^{\psi}\, \sum_{k,\ell\in\mathbb{Z}} \Big( \nu_{+,k} \wedge \bar{\nu}_{+,\ell} \Big)\, d^2z.
\end{split}
\end{equation}
Here the last equality holds since $\{\nu_{+,k}\}_{k\in\mathbb{Z}}$ is an orthonormal basis.
By these equalities and \eqref{e:var-nu+}, 
\begin{equation}\label{e:symplectic-form+1}
\begin{split}
&
-\frac{\sqrt{-1}}{2}\mathrm{Mess}^*_+\int_{\mathbb{D}}  \, e^{\psi}\, \Big( \nu_{+,k} \wedge \bar{\nu}_{+,\ell} 
+ \nu_{+,\ell} \wedge \bar{\nu}_{+,k} \Big)\, d^2z\\
=& -\frac{\sqrt{-1}}{2}\int_{\mathbb{D}} (1-|\mu_F|^2)\, e^{\phi}\\
&\qquad \cdot   \Big(2(1-|\mu_F|^2)^{-1} (\nu_k-\bar{\nu}_k\mu_F^2)\wedge
2(1-|\mu_F|^2)^{-1}(\bar{\nu}_\ell-\nu_\ell\bar{\mu}_F^2)\Big)\, d^2z\\
& -\frac{\sqrt{-1}}{2}\int_{\mathbb{D}} (1-|\mu_F|^2)\, e^{\phi}\\
&\qquad \cdot   \Big(2(1-|\mu_F|^2)^{-1} (\nu_\ell-\bar{\nu}_\ell\mu_F^2)\wedge
2(1-|\mu_F|^2)^{-1}(\bar{\nu}_k-\nu_k\bar{\mu}_F^2)\Big)\, d^2z\\
=& -2{\sqrt{-1}}\int_{\mathbb{D}} (1-|\mu_F|^2)^{-1}\, e^{\phi}\\
&\qquad\qquad \cdot \Big( \nu_k\wedge \nu_\ell +\nu_\ell\wedge \bar{\nu}_k +\bar{\nu}_k\wedge \nu_\ell\, |\mu_F|^4 +\bar{\nu}_\ell \wedge \nu_k\, |\mu_F|^4 \Big)\, d^2z\\
=&-2\sqrt{-1}\int_{\mathbb{D}} (1+|\mu_F|^2)\, e^\phi\, \Big( \nu_k\wedge \bar{\nu}_\ell + \nu_\ell\wedge \bar{\nu}_k \Big)\, d^2z.
\end{split}
\end{equation}
Combining \eqref{e:symplectic-form+} and \eqref{e:symplectic-form+1},
\begin{equation}\label{e:pullback-form}
-\sqrt{-1}\, i^*_{+} \omega_C= -2\int_{\mathbb{D}} (1+|\mu_F|^2)\, e^\phi\, \sum_{k,\ell\in \mathbb{Z}}
\Big( \nu_k\wedge \bar{\nu}_\ell \Big)\, d^2z.
\end{equation}
Hence, by \eqref{e:dd-bar} and \eqref{e:pullback-form}, we conclude
\begin{equation*}
\partial\bar{\partial} \Big(2 \int_{\mathbb{D}} e^\phi |\mu_F|^2 \,d^2z \Big) = -\sqrt{-1}\, i^*_+ \omega_{\mathrm{C}}
\qquad \text{over} \quad T_0(1)^+.
\end{equation*}
This completes the proof for the case of $T_0(1)^+$.
\end{proof}

By Theorems \ref{t:symplectic-relation} and \ref{t:K-potential}, we have the following result:

\begin{theorem}
Over the subspace $T_0(1)^\pm$ of $T^*T_0(1)$,
\begin{equation*}
\partial\bar{\partial}\, \big(2E\big)=\partial\bar{\partial} \Big(2 \int_{\mathbb{D}} e^{\phi}|\mu_F|^2\, d^2z \Big)=\sqrt{-1}\, i^*_{\pm}\mathrm{Mess}^*_\pm \big(\omega_{\mathrm{WP}}\big).
\end{equation*}
\end{theorem}

\begin{remark}
By Remark \ref{r:pullback-form} and the proof of Theorem \ref{t:K-potential}, the pullback 2-form $\mathrm{Mess}^*_{\pm}(\omega_{\mathrm{WP}})$
to $T_0(1)^\pm$ has the following expression:
\begin{equation}
 i^*_{\pm}\mathrm{Mess}^*_{\pm}(\omega_{\mathrm{WP}})= 2{\sqrt{-1}} \int_{\mathbb{D}} (1+|\mu_F|^2) e^\phi \,  \sum_{k,\ell\in\mathbb{Z}}\,
\nu_k\wedge \bar{\nu}_\ell \ d^2z
\end{equation}
where $\{\nu_k\}_{k\in\mathbb{Z}}$ is a basis of $T_{[\mu]} T_0(1)$. 

\end{remark}

\begin{remark}\label{r:relation-E-S}
From equation \eqref{e:dd-bar}, it follows that the holomorphic derivative of the anti-holomorphic energy functional $E$,  associated with the harmonic map $F_\pm$,  on $T_0(1)^\pm$ is given by the Hopf differential $\Phi(F_\pm)$.  
In contrast,  the holomorphic derivative of the universal Liouville action $S$ on $T_0(1)$ is expressed in terms of
the Schwarzian of a univalent function on $\mathbb{D}$, determined by the conformal welding data,  as established  in Theorem 3.1 of \cite{TT06}.
Understanding  the difference of these two holomorphic quadratic differentials is therefore essential for
the analyzing  the difference between $E$ and $S=\pi I_L$ ,  an issue that will be  addressed in future work.
\end{remark}

\appendix

\section{Anti de Sitter space of dimension 3}

Let $\mathbb{R}^{2,2}$ denote the pseudo-Euclidean 4-space with linear coordinates $\mathbf{x}=(x_1,x_2,x_3,x_4)$.  Consider the quadratic form 
\begin{equation}
q(\mathbf{x})= x^2_1+x^2_2-x^2_3-x^2_4,
\end{equation} 
and let $\langle \cdot, \cdot \rangle$ be the associated symmetric bilinear form.
The group $\mathrm{O}(2,2)$ consists of linear transformations of $\mathbf{R}^{2,2}$ that preserve $q$.
We define the hyperboloid
\begin{equation}
\mathbb{H}^{2,1}=\{\, \mathbf{x}\in \mathbb{R}^{2,2}\, | \, q(\mathbf{x}) =-1\, \}.
\end{equation}
One can verify that $\mathbb{H}^{2,1}$ is a smooth connected $3$-dimensional submanifold of $\mathbb{R}^{2,2}$. 
The tangent space $T_{\mathbf{x}} \mathbb{H}^{2,1}$ at $x\in \mathbb{H}^{2,1}$ is identified with the subspace
\begin{equation*}
\mathbf{x}^{\perp}=\{\mathbf{y}\in\mathbb{R}^{2,2}\, |\, \langle \mathbf{x},\mathbf{y} \rangle =0\}.
\end{equation*}
The restriction of  $\langle \cdot, \cdot \rangle$ to $T\mathbb{H}^{2,1}$ has the Lorentzian signature $(2,1)$,  making $\mathbb{H}^{2,1}$ a Lorentzian manifold. 

The $3$-dimensional \textit{Anti-de Sitter space} is then defined as
\begin{equation}\label{e:AdS-def}
\mathbb{A}\mathrm{d}\mathbb{S}^{2,1}:=\mathbb{H}^{2,1}/\{\pm \mathrm{Id}\},
\end{equation}
where $\mathrm{Id}$ is the identity element in $\mathrm{O}(2,2)$. The space $\mathbb{A}\mathrm{d}\mathbb{S}^{2,1}$ inherits the Lorentzian metric from $\mathbb{H}^{2,1}$ and has the constant curvature $-1$. By the definition, $\mathbb{A}\mathrm{d}\mathbb{S}^{2,1}$ can be identified with a subset of  the real projective space $\mathbb{RP}^{3}$:
\begin{equation*}
\mathbb{A}\mathrm{d}\mathbb{S}^{2,1}=\{ [\mathbf{x}]\in \mathbb{RP}^3\, | \, q(\mathbf{x})<0 \}. 
\end{equation*}
The boundary of $\mathbb{A}\mathrm{d}\mathbb{S}^{2,1}$ in $\mathbb{RP}^3$ is the projectivization of the set of lightlike vectors in $\mathbb{R}^{2,2}$,
\begin{equation}\label{e:def-AdS}
\partial \mathbb{A}\mathrm{d}\mathbb{S}^{2,1}=\{  [\mathbf{x}]\in \mathbb{RP}^3\, | \, q(\mathbf{x})=0 \}.
\end{equation}

Let $\mathrm{M}(2,\mathbb{R})$ denote the vector space of $2\times2$  real matrices. There is an isometric identification between $(\mathrm{M}(2,\mathbb{R}), -\mathrm{det})$
and $(\mathbb{R}^{2,2},q)$,
under which  the hyperboloid $\mathbb{H}^{2,1}$ corresponds to the special linear group $\mathrm{SL}(2,\mathbb{R})$.
The Lie group $\mathrm{SL}(2,\mathbb{R})$ has a bi-invariant bilinear form, known as the Killing form $\kappa$,  on its Lie algebra 
$\mathfrak{sl}(2,\mathbb{R})$. The Killing form $\kappa$ has the Lorentzian signature $(2,1)$, inducing a Lorentzian metric on $\mathrm{SL}(2,\mathbb{R})$, which we denote by 
$g_{\kappa}$.  From \eqref{e:AdS-def}, one can verify that 
the Anti-de Sitter space $\mathbb{A}\mathrm{d}\mathbb{S}^{2,1}$ is naturally identified with $\mathrm{PSL}(2,\mathbb{R})$ equipped with  the Lorentzian metric $\frac18 g_{\kappa}$. 

The group $\mathrm{SL}(2,\mathbb{R})\times\mathrm{SL}(2,\mathbb{R})$ acts on $\mathrm{M}(2,\mathbb{R})$ via 
\begin{equation*}
 (\alpha,\beta)\cdot\gamma= \alpha\circ\gamma \circ \beta^{-1}\qquad \text{for} \quad \gamma\in  \mathrm{M}(2,\mathbb{R}).
\end{equation*}
This action preserves the quadratic form $-\mathrm{det}\cong q$, leading to the identification
\begin{equation*}
\mathrm{Isom}_0(\mathbb{H}^{2,1})\cong \mathrm{SO}_0(\mathrm{M}(2,\mathbb{R}),q) \cong \large(\mathrm{SL}(2,\mathbb{R})\times \mathrm{SL}(2,\mathbb{R}) \large)/K,
\end{equation*}
where $K:=\{(\mathrm{Id},\mathrm{Id}),(-\mathrm{Id},-\mathrm{Id})\}$.  Consequently, the connected component of the isometry group of $\mathbb{A}\mathrm{d}\mathbb{S}^{2,1}$ is given by

\begin{equation*}
\mathrm{Isom}_0(\mathbb{A}\mathrm{d}\mathbb{S}^{2,1})\cong \mathrm{PSL}(2,\mathbb{R})\times \mathrm{PSL}(2,\mathbb{R}).
\end{equation*}

By \eqref{e:def-AdS},  the boundary of  $\mathbb{A}\mathrm{d}\mathbb{S}^{2,1}$  in the projectivized space  $ \mathrm{P}(\mathrm{M}(2,\mathbb{R}))$ is given by
\begin{equation}
\partial \mathbb{A}\mathrm{d}\mathbb{S}^{2,1}=\{  [X]\in \mathrm{P}(\mathrm{M}(2,\mathbb{R})) \, | \, \mathrm{rank}(X)=1 \}.
\end{equation}
This boundary admits the following homeomorphism:
\begin{equation}
\partial \mathbb{A}\mathrm{d}\mathbb{S}^{2,1} \, \to \, \mathbb{R}\mathrm{P}^1\times\mathbb{R}\mathrm{P}^1,
\end{equation}
which is explicitly defined by
\begin{equation}
[X] \mapsto (\mathrm{Im}(X),\mathrm{Ker}(X)).
\end{equation}
Timelike geodesics in  $\mathbb{A}\mathrm{d}\mathbb{S}^{2,1}$ are given by
\begin{equation*}
L_{p,q}=\{ \gamma\in \mathrm{PSL}(2,\mathrm{R})\, |\, \gamma(q)= p\,\}
\end{equation*}
for some points $p,q\in \mathbb{D}$.

For a spacelike conformal embedding $\sigma:\mathbb{D}\to \mathbb{A}\mathrm{d}\mathbb{S}^{2,1}$, 
let $\Sigma$ denote the image $\sigma(\mathbb{D})\subset  \mathbb{A}\mathrm{d}\mathbb{S}^{2,1}$. 
The associated \textit{Gauss map} 
\begin{equation}\label{e:A-Gauss}
G:\Sigma\to \mathbb{D}\times\mathbb{D}
\end{equation} 
is defined by
$$
G(x)=(p,q),
$$ where $L_{p,q}$ is the timelike geodesic orthogonal to the tangent space  of $\Sigma$ at $x$.

\end{document}